\documentclass[11pt]{amsart}  

\usepackage{amssymb}
\usepackage{amsmath}
\usepackage{enumitem}
\usepackage{mathrsfs}
\usepackage[english]{babel}
\usepackage[all]{xy}
\usepackage{hyperref}
\usepackage{marginnote}
\usepackage{fullpage}
\usepackage{stmaryrd}

\usepackage{tikz-cd}
\usepackage{tikz}

\usepackage[margin=1.2in]{geometry}

\newtheorem{theorem}{Theorem}[section]
\newtheorem{lemma}[theorem]{Lemma}
\newtheorem{proposition}[theorem]{Proposition}
\newtheorem{corollary}[theorem]{Corollary}
\newtheorem{conjecture}[theorem]{Conjecture}

\theoremstyle{definition}
\newtheorem*{ack}{Acknowledgements}

\newtheorem{remark}[theorem]{Remark}
\newtheorem{example}[theorem]{Example}
\newtheorem{definition}[theorem]{Definition}

\newtheorem{alphtheorem}{Theorem}

\numberwithin{equation}{section} \numberwithin{figure}{section}

\DeclareMathOperator{\Spec}{Spec}

\DeclareMathOperator{\an}{an}

\DeclareMathOperator{\Hom}{Hom}

\DeclareMathOperator{\ev}{ev}

\DeclareMathOperator{\supp}{supp}

\newcommand*\ratmap{\mathbin{\tikz [baseline=-0.25ex,-latex, dashed, ->] \draw [densely dashed] (0pt,0.5ex) -- (1.3em,0.5ex);}}

\newcommand{\et}{\textrm{\'{e}t}}

\usepackage{color}

\definecolor{orange}{rgb}{1,0.5,0}

\title[Weakly-special threefolds and non-density of rational points]{Weakly-special threefolds and non-density of rational points}

\author{Finn Bartsch}
\address{Finn Bartsch \\
IMAPP Radboud University Nijmegen \\
PO Box 9010, 6500GL \\
Nijmegen, The Netherlands\\}
\email{f.bartsch@math.ru.nl }

\author{Ariyan Javanpeykar}
\address{Ariyan Javanpeykar \\ 
IMAPP Radboud University Nijmegen \\
PO Box 9010, 6500GL\\
Nijmegen, The Netherlands}
\email{ariyan.javanpeykar@ru.nl }

\author{Erwan Rousseau}
\address{Erwan Rousseau \\ Univ Brest\\
		CNRS UMR 6205
		\\Laboratoire de Mathematiques de Bretagne Atlantique\\ F-29200 Brest, France
		}
\email{erwan.rousseau@univ-brest.fr}

\subjclass[2010]
{14G99 
(11G35,  
14G05,  
32Q45)} 

\keywords{hyperbolicity, function fields, rational points, Lang--Vojta conjecture}

\begin{document}

\begin{abstract}   
We verify   part of a conjecture of Campana predicting that   rational points on the weakly-special non-special simply-connected smooth projective threefolds constructed by Bogomolov--Tschinkel are not dense.   To prove our result, we establish fundamental properties of moduli spaces of orbifold maps, and prove a dimension bound for such moduli spaces by using the recent extension of Kobayashi--Ochiai's finiteness theorem for Campana's orbifold pairs. 
\end{abstract}

\maketitle
 
\thispagestyle{empty}

\section{Introduction}    
A variety over a field $k$ is a geometrically integral finite type separated scheme over $k$. Lang's conjecture on rational points over number fields \cite[Conjecture~5.7]{Langconj} predicts that, for a smooth projective variety $X$ of general type over a number field $K$, the set $X(L)$ of $L$-points on $X$ is not dense for every finitely generated field extension $L/K$.  In this paper we are concerned with a class of varieties ``opposite'' to the class of varieties of general type.

\begin{definition}
A smooth projective variety $X$ over a field $k$   is \emph{weakly-special} if no finite \'etale cover of $X_{\overline{k}}$ admits a dominant rational map $X_{\overline{k}}\ratmap Y$ to   a positive-dimensional smooth projective variety of general type $Y$ over $\overline{k}$,  where $\overline{k}$ is an algebraic closure of $k$.
\end{definition}

Let $X$ be a projective variety over a finitely generated field $K$ of characteristic zero. Lang's conjecture combined with the Chevalley-Weil theorem predicts that, if $X(K)$ is dense, then $X$ is weakly-special.  
In \cite[Conjecture~1.2]{HarrisTschinkel} the converse of this statement was conjectured over number fields. The following conjecture is its natural extension to all finitely generated fields of characteristic zero.

\begin{conjecture} \label{conj}
If $X$ is a weakly-special smooth projective geometrically connected variety over a finitely generated field $K$ of characteristic zero, then there exists a finite field extension $L/K$ such that $X(L)$ is dense in $X$.
\end{conjecture}

This conjecture  is incompatible with a series of conjectures introduced by Campana in his foundational work on special varieties  \cite[\S~13.6]{Ca11}. Indeed, Campana conjectured that $X$ is special if and only if there exists a finite field extension $L/K$ such that $X(L)$ is dense in $X$.  However, Bogomolov--Tschinkel \cite{BT} constructed examples of weakly-special non-special varieties, so that the two conjectures are in conflict. We formalize their construction as follows.

\begin{definition}\label{defn:bt}
If $k$ is a field, then a $k$-scheme $X$ is called a \emph{BT-threefold} if it fits into a cartesian square
\begin{equation*} \begin{tikzcd}
X \ar[r] \ar[d] & S \ar[d, "\psi"] \\ B \ar[r, "\phi"] & \mathbb{P}^1
\end{tikzcd} \end{equation*}
where 
\begin{itemize}
\item $S$ is a smooth projective surface and $\psi$ is a non-isotrivial elliptic fibration.
\item The fiber $\psi^{-1}(0)$ is a multiple fiber with multiplicity $m$, and this is the only multiple fiber. 
\item $B$ is a smooth projective surface of Kodaira dimension $1$ whose associated elliptic fibration is non-isotrivial.
\item The map $\phi$ has fibers of genus $g \geq 2$.
\item The fiber $D := \phi^{-1}(0)$ is smooth.
\item $B \setminus D$ is simply-connected.
\item The singular loci of $\phi$ and $\psi$ are disjoint. 
\end{itemize}
Additionally, a BT-threefold is a \emph{BTCP-threefold} if
\[
c_1(B)^2 + (1-\tfrac{1}{m})K_B\cdot D > c_2(B) + (\tfrac{1}{m}-\tfrac{1}{m^2})D^2.
\]
If $X$ is a BT-threefold, then we will refer to the morphism $X \to B$ as the \emph{associated elliptic fibration} and define the $\mathbb{Q}$-divisor $\Delta := \left(1-\frac{1}{m}\right) D.$
\end{definition}

Our starting point in this paper is the culmination of the work of Bogomolov--Tschinkel \cite{BT}, Campana--P\u{a}un \cite{CampanaPaun2007} and \cite{CampanaWinkelmannBrody}; see Section \ref{section:bt_threefolds} for a detailed discussion.

\begin{theorem}[Bogomolov--Tschinkel, Campana--P\u{a}un, Campana--Winkelmann]
The following statements hold. 
\begin{enumerate}
\item There exist BTCP-threefolds over $\mathbb{Q}$.
\item A BT-threefold over $\mathbb{C}$ is a simply-connected smooth projective weakly-special variety which is not special.
\item If $X$ is a BTCP-threefold over $\mathbb{C}$, then $X^{\an}$ does not have a dense entire curve.
\item There exist BTCP-threefolds $X$ over $\mathbb{C}$ such that the Kobayashi pseudometric on $X^{\an}$  does not vanish identically.
\end{enumerate}
\end{theorem}

Our main result is concerned with the distribution of rational points on BTCP-threefolds over function fields. By Campana's conjectures, they should not be dense over any given finitely generated field of characteristic zero. We prove part of this function field analogue of the above analytic results in Section \ref{section:bt_threefolds}.

\begin{alphtheorem} \label{thm:main_final}
Let $X$ be a BTCP-threefold over a finitely generated field $K$ of characteristic zero and let $V$ be a variety over $K$. Then the set of non-constant rational maps $V \ratmap X$ is not dense in $X_{K(V)}$.   
\end{alphtheorem} 

Theorem \ref{thm:main_final} can be reformulated geometrically as follows: if $X$ is a BTCP-threefold over a finitely generated field $K$ of characteristic zero and $V$ is a variety over $K$, then the union $\cup_f \Gamma_f$ of graphs $\Gamma_f$ with $f\colon V\ratmap X$ a non-constant rational map is not dense in $X\times V$.

For $X$ as in Theorem \ref{thm:main_final}, the natural extension of Campana's arithmetic conjectures (see \cite[Conjecture~1.6]{JR}) predicts that $X(K(V))$ is not dense in $X$. If we define $X(K(V))^{nc}$ to be the subset of $K(V)$-points which define a non-constant rational map $V \ratmap X$, then our main result only guarantees the non-density of the subset $X(K(V))^{nc}$ in $X_{K(V)}$.

Assuming Vojta's higher-dimensional abc conjecture \cite{VojtaIMRN}, it can be deduced from \cite[Corollary~3.3]{AbramovichVA} that the rational points on a BTCP-threefold over a number field $K$ are not dense.  By constructing new weakly-special non-special threefolds,  it was shown  \cite{BCJW}  that  Conjecture \ref{conj} contradicts the abc conjecture.
 
  Theorem \ref{thm:main_final} is one of the few results we have concerning the geometric discrepancy between special and weakly-special varieties. Its proof relies on Faltings's finiteness theorem for rational points on higher genus curves over number fields. Indeed, in the case that $V$ is a smooth projective curve, Theorem \ref{thm:main_final} can be reformulated as saying that the ``moduli space of non-constant orbifold maps from $V$ to the orbifold base $(B,\Delta)$ of $X\to B$'' has only finitely many $K$-rational points; see Section \ref{section:orbifolds} for the definition of the orbifold base.  A large part of our proof is thus dedicated to proving that positive-dimensional irreducible components of such moduli spaces are birational to higher genus curves (see Theorem \ref{thm:main2} for a precise statement).

Note that, using different ideas and building on work of Corvaja--Zannier \cite{CZ04} and Ru--Vojta \cite{RV}, other examples of weakly-special non-special threefolds  $X'$ were constructed in \cite{RTW}. However,  the non-density of rational points over function fields on the threefolds constructed in \emph{loc. cit.}   (i.e., the analogue of Theorem \ref{thm:main_final}) remains   unknown, even though they also come with an elliptic fibration $X'\to B'$ onto a surface $B'$.  The main difficulty in extending our work to those threefolds is that $B'$   in their examples is uniruled, and  our proof of Theorem \ref{thm:main_final} heavily relies on the results  in Section \ref{section:dimension_of_moduli}  which only apply to non-uniruled surfaces.

\subsection{Outline of paper}
To prove Theorem \ref{thm:main_final}, we first reduce to the case that $V$ is a smooth projective curve using a standard cutting argument (Lemma \ref{lemma:cutting_argument}). Let $X\to B$ be the elliptic fibration on the BTCP-threefold $X$. Since this fibration is non-isotrivial, almost all of the composed maps $V \to X \to B$ will be non-constant by the isogeny theorem for elliptic curves over number fields (Lemma \ref{lemma:isotriviality}). Thus, we are naturally led to studying curves on the Kodaira dimension one surface $B$.  

At this point we  note that the curves in $B$ which ``come from $X$'' satisfy tangency conditions with respect to the divisor $D$. To keep track of these tangency conditions, it is natural to use Campana's notion of an orbifold pair and an orbifold morphism (see Section \ref{section:orbifolds}). The main   finiteness result we prove for orbifolds in this paper is Theorem \ref{thm:main2}.  

Our proof of Theorem \ref{thm:main_final}  relies on a study of the subset $\mathcal{H}$ of non-constant maps from a curve $C$ to a surface $B$ satisfying Campana-like tangency conditions inside the Hom-scheme parametrizing all morphisms from $C$ to $B$; we show that $\mathcal{H}$ is naturally a locally closed subset and that the universal evaluation map $C \times \mathcal{H} \to B$  likewise satisfies similar tangency conditions with respect to $D$ (Theorem \ref{thm:algebraicity_of_moduli}). That is, the universal evaluation map ``inherits'' the orbifold properties from the maps it parametrizes. It is this latter property which is central to our proof.  In its proof we use the fact that, given a family of orbifold morphism $(f_t \colon C \to (B,\Delta))_{t\in T}$ with $T$ a smooth parameter space, the total morphism $C\times T \to B$ is again an orbifold morphism $C\times T\to (B,\Delta)$; see Theorem \ref{thm:families_of_orbifold_maps} for a precise statement.

To prove the non-density of the non-constant $K(V)$-points on $X$, we are thus led to study Mordell-type properties of the moduli space $\mathcal{H}$. In fact, we seek to establish the finiteness of its $K$-rational points for every finitely generated field $K$ of characteristic zero. The finiteness of $K$-rational points of $\mathcal{H}$ is achieved in three steps:

\begin{enumerate}
\item By the theory of Hilbert schemes, $\mathcal{H}$ is a priori a countable union of quasi-projective schemes. The first step of our proof consists of showing that the moduli space $\mathcal{H}$ is in fact quasi-projective (i.e., of finite type). We do this by generalizing Bogomolov's theorem on surfaces of general type with positive second Segre class to the setting of Campana's orbifold pairs (proven in Section \ref{section:bogomolovs_thm}, where   the terminology is also explained). \end{enumerate}

\begin{alphtheorem}[Bogomolov's theorem for orbifold surfaces]\label{thm:bogomolov}
If $(X,\Delta)$ is a smooth projective orbifold surface of general type  and $c_1(X,\Delta)^2>c_2(X,\Delta)$, then $(X, \Delta)$ is pseudo-algebraically hyperbolic over $k$.
\end{alphtheorem} 

\begin{enumerate}[resume]
\item The second  observation is that the dimension of $\mathcal{H}$ is at most one (Corollary \ref{corollary:boundedness}). Here we apply the recently established orbifold extension of the theorem of Kobayashi--Ochiai  \cite[Theorem~1.1]{BJ}. Namely, a given variety $Y$ admits only finitely many surjective morphisms to a given orbifold $(B,\Delta)$ of general type; see the proof of Lemma \ref{lemma:ko_app} for details. At this point we know that $\mathcal{H}$ is at most one-dimensional.
\item To prove finiteness of $K$-points on $\mathcal{H}$, we have to exclude the possibility that $\mathcal{H}$ has components of genus zero or one, so that we can appeal to Faltings's theorem for hyperbolic curves (\emph{formerly} Mordell's conjecture). This part of the argument  uses that $B$ is a Kodaira dimension one surface whose associated elliptic fibration is non-isotrivial (whereas the previous steps can be performed in a far more general setting); see Corollary \ref{cor:finiteness_for_orbifold_surface1} for a precise statement.
\end{enumerate}
 
\begin{ack}
We are grateful to Remke Kloosterman for helpful discussions on elliptic fibrations  and, in particular, for    Remark \ref{remark:kloosterman}.
\end{ack}

\section{Campana's orbifold pairs} \label{section:orbifolds}
Let $k$ be a field of characteristic zero. Recall that a variety over $k$ is a geometrically integral finite type separated scheme over $k$.  

\begin{definition}
A \emph{$\mathbb{Q}$-orbifold (over $k$)} $(X, \Delta)$ is a variety $X$ together with a $\mathbb{Q}$-Weil divisor $\Delta$ on $X$ such that all coefficients of $\Delta$ are in $[0,1]$. If $\Delta = \sum_i \nu_i \Delta_i$ is the decomposition of $\Delta$ into prime divisors, we say that $m(\Delta_i) := (1-\nu_i)^{-1}$ is the \emph{multiplicity} of $\Delta_i$ in $\Delta$. If all multiplicities of a $\mathbb{Q}$-orbifold are in $\mathbb{Z} \cup \{ \infty\}$, we say that $(X,\Delta)$ is an \emph{orbifold}.
\end{definition}

An important class of orbifold divisors are those associated to a fibration with multiple fibers (see \cite[Definition~4.2]{Ca11}:   

\begin{definition}[Orbifold base]\label{definition:orbifold_base}  
Let $f \colon X \to Y$ be a surjective morphism of normal varieties over $k$ with geometrically connected fibers. Assume $Y$ is smooth.  Let $D \subset Y$ be a prime divisor of $Y$, and note that   $f^\ast D$   is a non-empty divisor on $X$ and that  we may write  $f^\ast D = R+\sum_{i} t_i\cdot F_i $, 
where $R$ is an effective $f$-exceptional divisor (i.e., none of its components surject onto $D$), each $F_i$ is an irreducible component of $X_{D}$ surjecting onto $D$, and $t_i\in \mathbb{Z}_{\geq 1}$.   Let  $m_f(D) := \inf \{t_i\}.$  We define the 
\emph{orbifold divisor} $\Delta_f$ \emph{of $f$} to be $$\Delta_f := \sum_{D} \left(1-\frac{1}{m_f(D)}\right) D,$$ where the sum runs over all prime divisors of $Y$. We refer to the orbifold $(Y,\Delta_f)$ as the \emph{orbifold base of} $f$.
\end{definition}

An orbifold $(X, \Delta)$ is an \emph{orbifold curve} (resp. \emph{orbifold surface}) if $\dim X = 1$ (resp. $\dim X = 2$).
A $\mathbb{Q}$-orbifold $(X, \Delta)$ is \emph{smooth} if the underlying variety $X$ is smooth and the support of the orbifold divisor $\supp \Delta$ is a divisor with strict normal crossings. It is \emph{normal} (resp. \emph{locally factorial}, resp. \emph{proper}, resp. \emph{projective}) if $X$ is normal (resp. locally factorial, resp. proper over $k$, resp. projective over $k$).

\begin{definition}[Morphisms]
Let $(X, \Delta_X)$ be a normal $\mathbb{Q}$-orbifold and $(Y, \Delta_Y)$ be a locally factorial $\mathbb{Q}$-orbifold. In this case, we define a \emph{morphism} of $\mathbb{Q}$-orbifolds $f \colon (X, \Delta_X) \to (Y, \Delta_Y)$ to be a morphism of varieties $f \colon X \to Y$ satisfying $f(X) \nsubseteq \supp \Delta_Y$ such that, for every prime divisor $E \subseteq \supp \Delta_Y$ and every prime divisor $D \subseteq \supp f^\ast E$, we have $t m(D) \geq m(E)$, where $t \in \mathbb{Q}$ denotes the coefficient of $D$ in $f^\ast E$; the local factoriality of $Y$ ensures that $E$ is a Cartier divisor, so that $f^\ast E$ is well-defined.
\end{definition}

If $X$ is a normal variety, we identify $X$ with the orbifold $(X, 0)$. If $X$ and $Y$ are varieties such that $X$ is normal and $Y$ is locally factorial, every morphism of varieties $X \to Y$ is an orbifold morphism $(X,0) \to (Y,0)$.

We will be interested in rational maps satisfying the orbifold condition. To make this precise, we follow \cite{BJ} and work with orbifold near-maps.

\begin{definition}
An open subscheme $U \subseteq X$ of a variety $X$ is \emph{big} if its complement is of codimension at least two. A rational map $X \ratmap Y$ of varieties is a \emph{near-map} if there is a big open $U \subseteq X$ such that $U \ratmap Y$ is a morphism. 
\end{definition}

Note that a rational map $X \ratmap Y$ is a near-map if and only if it is defined at all codimension one points of $X$. For example, for every normal variety $X$ and any proper variety $Y$, every rational map $X \ratmap Y$ is a near-map.

\begin{definition}
Let $(X, \Delta_X)$ be a normal $\mathbb{Q}$-orbifold and $(Y, \Delta_Y)$ be a $\mathbb{Q}$-orbifold such that $Y$ is locally factorial. An \emph{orbifold near-map} \[f \colon (X, \Delta_X) \ratmap (Y, \Delta_Y)\] is a near-map $f \colon X \ratmap Y$ satisfying $f(X) \nsubseteq \supp \Delta_Y$ such that, for every prime divisor $E \subseteq \supp \Delta_Y$ and every prime divisor $D \subseteq \supp f^*E$, we have $t m(D) \geq m(E)$, where $t \in \mathbb{Q}$ denotes the coefficient of $D$ in $f^*E$; this pullback is well-defined as $E$ is Cartier.
\end{definition}

\subsection{Chern classes}
Let $(X,\Delta)$ be a smooth projective orbifold surface. To state our main result, we will need to define the Chern classes of $(X,\Delta)$. Let $\Delta = \sum (1-\frac{1}{m_i})\Delta_i$ be the decomposition into irreducible components. Let $\mathrm{CH}(X)$ be the Chow ring of $X$ and let $\mathrm{CH}^d \subset \mathrm{CH}(X)$ denote the subgroup of codimension $d$ cycle classes.
\begin{definition}\label{definition:chernclasses}
We define 
\begin{align*}
\mathbf{c_1}(X,\Delta) &:= -(K_X + \Delta) && \text{in } \mathrm{CH}^1(X) \otimes \mathbb{Q} \\ 
\mathbf{c_2}(X,\Delta) &:= \mathbf{c_2}(X) + \sum_i \left(1-\frac{1}{m_i}\right)(K_X+\Delta_i).\Delta_i + \sum_{i < j} \left(1-\frac{1}{m_i m_j}\right)\Delta_i.\Delta_j &&\text{in } \mathrm{CH}^2(X) \otimes \mathbb{Q}.
\end{align*}
\end{definition}

We will be interested in the degrees of $\mathbf{c_1}^2$ and $\mathbf{c_2}$. Thus, we define:
\[ c_1(X,\Delta)^2 := \deg \mathbf{c_1}(X,\Delta)^2 \quad \textrm{and} \quad c_2(X,\Delta) := \deg \mathbf{c_2}(X,\Delta) \quad \text{in }\mathbb{Q}. \]

If $\Delta$ is the trivial divisor, then $\mathbf{c_1}(X,\Delta) = \mathbf{c_1}(X)$ and $\mathbf{c_2}(X,\Delta) = \mathbf{c_2}(X)$, where $\mathbf{c_i}(X)$ is the $i$-th Chern class of $T_X$, so that we recover the ``usual'' Chern classes of $X$. If the multiplicities of $\Delta$ are all infinite, then $\mathbf{c_i}(X,\Delta) = \mathbf{c_i}(T_X(-\log \Delta))$, so that we recover the Chern classes of the log-pair $(X,\Delta)$. If $D$ is a prime divisor on $X$ and $\Delta = \left(1-\frac{1}{m}\right)D$ with $m\geq 1$, then the (numerical) Chern classes of $(X,\Delta)$ defined above are the (numerical) Chern classes of $T_{\mathcal{X}}$, where $\mathcal{X} := \sqrt[m]{X/D}$ is the $m$-th root stack of $X$ along $D$.

\begin{example}\label{example:btcp}
Let $X$ be a BT-threefold (Definition \ref{defn:bt}) with associated elliptic fibration $X\to B$ and write $\Delta = (1-\frac{1}{m}) D$. Then $(B,\Delta)$ is the orbifold base of $X\to B$ and $X$ is a BTCP-threefold if and only if 
\[ c_1\left(B,\Delta \right)^2 > c_2\left(B,\Delta \right). \] 
\end{example}

\subsection{Main result}
Our main   finiteness result for orbifold surfaces is as follows (and is proven in Section \ref{section:proof_of_thm}).  

\begin{alphtheorem}\label{thm:main2} 
Let $(B,\Delta)$ be a smooth projective orbifold of general type over a finitely generated field $K$ of characteristic zero, where $B$ is a Kodaira dimension one surface with non-isotrivial elliptic fibration. If $c_1(B,\Delta)^2 > c_2(B,\Delta)$, then there is a proper closed subset $Z \subsetneq B$ such that, for every finitely generated field extension $L/K$ and every variety $V$ over $L$, the set of non-constant near-maps $f \colon V \ratmap (B_L,\Delta_L)$ with $f(V) \not\subset Z$ is finite.
\end{alphtheorem}
 
Theorem \ref{thm:main2} is a Mordellicity statement (i.e., a finiteness result for rational points) for moduli spaces of orbifold maps. It is a natural orbifold extension of \cite[Theorem~1.3]{JMRL}, and implies that $(X,\Delta)$ is pseudo-$p$-Mordellic for $p>0$ (where we freely adapt the terminology in \cite[Section~2]{EJR} to the orbifold setting). Its proof is a mixture of algebro-geometric and arithmetic arguments. The geometric ingredients of its proof include the recent extension of Kobayashi--Ochiai's finiteness theorem for dominant maps to a variety of general type (Theorem \ref{thm:ko}) and an orbifold extension of Bogomolov's theorem for surfaces with positive second Segre class (Theorem \ref{thm:bogomolov}). We will also need two arithmetic finiteness results: (1) Faltings's proof of the Mordell conjecture and (2) Shafarevich's isogeny theorem for elliptic curves.

We stress that Theorem \ref{thm:main2} is \emph{false} over algebraically closed fields of characteristic zero, i.e., the assumption that $K$ is finitely generated can not be omitted. This is simply because there are surfaces $B$ as in Theorem \ref{thm:main2} which are dominated by a product of (higher genus) curves; see Remark \ref{remark:kloosterman} for an explicit example.

Theorem \ref{thm:main2} is reminiscent of the theorem of De Franchis that, for $V$ and $X$ varieties, the set of non-constant morphisms $V \to X$ is finite when $X$ is a log-general type curve. Such finiteness results pertain to the finiteness of certain Hom-schemes, whereas Theorem \ref{thm:main2} only guarantees the Mordellicity of the relevant Hom-schemes (as the desired zero-dimensionality can certainly fail as we just explained).

As we briefly explained in the introduction, Theorem \ref{thm:main2} is used to prove Theorem \ref{thm:main_final}. In fact, for the threefolds $X$ considered in Theorem \ref{thm:main_final}, there is a smooth projective surface $B$ of Kodaira dimension one and an elliptic fibration $X\to B$ whose orbifold base $(B,\Delta)$ (see Definition \ref{definition:orbifold_base}) is of general type and satisfies $c_1(B,\Delta)^2 > c_2(B,\Delta)$.  Now,  we observe that almost all of the non-constant rational maps   $V\ratmap X$   considered in Theorem \ref{thm:main_final} give rise to non-constant orbifold near-maps $V \ratmap (B,\Delta)$, and the latter are finite modulo some exceptional locus by Theorem \ref{thm:main2}. We refer to Section \ref{section:bt_threefolds} for details.


\section{The moduli space of orbifold maps}\label{section:moduli_of_orbifold_maps}

We show that the subset of orbifold maps inside the moduli space of maps from a fixed curve to an orbifold defines a locally closed subscheme (see Corollary \ref{cor:hom_curve_to_orbifold}). We deduce this from another result on  families of orbifold maps (Theorem \ref{thm:families_of_orbifold_maps}).

\subsection{Vanishing of sections of line bundles}
In this section we prove the following presumably well-known result; due to lack of reference we include a proof. We stress that in this statement the scheme $X$ is assumed to have no embedded points \cite[Tag~05AK]{stacks-project}, but may be very well nonreduced. 
 
\begin{proposition}\label{proposition:vanishing_of_t}
Let $f \colon X \to S$ be a finite type dominant morphism of schemes, where $X$ is an irreducible scheme with no embedded points and $S$ is an integral noetherian scheme. Let $\mathcal{L}$ be a line bundle on $X$ and let $t\in \mathcal{L}(X)$ be a global section such that, for a dense set of points $s \in S$, the restriction of $t$ to $X_s$ vanishes. Then $t=0$.
\end{proposition}

We proceed in two steps. First, we show vanishing of $t$ on the generic fiber and then use this to show global vanishing.

\begin{lemma}[Vanishing on generic fiber]\label{lemma:generic_vanishing}
Let $f \colon X \to S$ be a finite type morphism of schemes, where $S$ is an integral noetherian scheme with generic point $\eta$. Let $\mathcal{L}$ be a line bundle on $X$ and let $t \in \mathcal{L}(X)$ be a global section such that, for a dense set of points $s \in S$, the restriction of $t$ to $X_s$ vanishes. Then the restriction of $t$ to the generic fiber $X_\eta$ vanishes. 
\end{lemma}
\begin{proof}
We can replace $S$ by a dense open $U \subseteq S$ as this preserves the hypothesis and does not change the generic fiber. Thus, we may assume that $S = \Spec A$ is affine. Clearly, by choosing an open affine covering of $X$, we may and do assume that $X = \Spec B$ is affine.

Note that $f$ induces a finite type morphism $\varphi \colon A \to B$ of commutative rings, where $A$ is a noetherian integral domain. We interpret the line bundle $\mathcal{L}$ as a projective $B$-module $M$ and the global section $t$ as an element $m$ of $M$. By generic freeness (\cite[Tag~051R]{stacks-project}), there is a nonzero element $f \in A$ such that $ M \otimes_A A[f^{-1}]$ is a free $A[f^{-1}]$-module. Replacing $A$ by a $A[f^{-1}]$ if necessary (i.e., $S$ by a dense open), we may assume that $M$ is a free $A$-module.

By assumption, there is an index set $I$ and prime ideals $\mathfrak{p}_i\subset A$ with $i\in I$ and $\cap_{i\in I} \mathfrak{p}_i = 0$ such that, for every $i$ in $I$, the element $m$ is in the kernel of $M \to M \otimes_A \kappa(\mathfrak{p}_i)$. However, the kernel of this map is $\mathfrak{p}_iM$. Moreover, the intersection satisfies $\cap_{i\in I} \mathfrak{p}_i M  = (\cap_{i\in I}\mathfrak{p}_i) M=0$ by freeness of $M$ over $A$. This implies that $m=0$, as required.
\end{proof}

Generic vanishing implies global vanishing, assuming in addition that the total space has no embedded points and that the morphism is dominant (but not necessarily of finite type).

\begin{lemma} [Global vanishing]\label{lemma:global_vanishing}
Let $f \colon X \to S$ be a dominant morphism of noetherian schemes. Assume that $S$ is integral and that $X$ is irreducible without embedded points. Let $\mathcal{L}$ be a line bundle on $X$ and $t \in \mathcal{L}(X)$. If the restriction of $t$ to the generic fiber $X_\eta$ is zero, then $t=0$.
\end{lemma}
\begin{proof}
The vanishing of a section can be tested locally, so that we may assume that $X$ and $S$ are affine and that $\mathcal{L}$ is trivial. In this case, the lemma reduces to the following statement:

Let $A \subseteq B$ be an inclusion of noetherian rings. Assume that $A$ is an integral domain and that $B$ has a unique minimal prime ideal and no embedded primes. Then the map $B \to (A \setminus \{0\})^{-1} B$ is injective. 

To prove this statement, let $b \in B$ be in the kernel of $B \to (A \setminus \{0\})^{-1} B$. Then there exists a nonzero $a \in A$ such that $ab=0$. Assume that $b \neq 0$. Then $a$ is a zerodivisor in $B$. However, since $B$ has no embedded primes, every zerodivisor in $B$ is nilpotent, so that $a$ is a nonzero nilpotent element. This contradicts the assumption that $A$ is an integral domain. Thus, we conclude that $b=0$, as required.   
\end{proof}

\begin{proof}[Proof of Proposition \ref{proposition:vanishing_of_t}]
Combine Lemma \ref{lemma:generic_vanishing} and Lemma \ref{lemma:global_vanishing}.
\end{proof}

\begin{remark}
The assumption on embedded points is necessary in Proposition \ref{proposition:vanishing_of_t} and Lemma \ref{lemma:global_vanishing}. Indeed, consider $X=\Spec k[x,y]/(xy,y^2)$, $S=\Spec k[x]$, and the finite surjective map $f \colon \Spec k[x,y]/(xy,y^2) \to \Spec k[x]$. Note that the nonzero element $y \in k[x,y]/(xy,y^2)$ (regarded as a section of $\mathcal{O}_X$) vanishes after tensoring with $k(x)$, i.e., it vanishes generically without vanishing globally. The problem in this situation is that the nilpotent $y$ is killed by the non-nilpotent $x$ (so that $x$ is a non-nilpotent zerodivisor in $\mathcal{O}(X)$). 
\end{remark}

\subsection{Families of orbifold maps}

\begin{theorem}\label{thm:families_of_orbifold_maps}
Let $k$ be an algebraically closed field of characteristic zero.
Let $\pi \colon \mathcal{X} \to S$ be a dominant morphism of smooth varieties over $k$ with geometrically integral fibers.
Let $(Y, \Delta_Y)$ be a locally factorial orbifold pair and let $f \colon \mathcal{X} \to Y$ be a morphism of varieties.
Assume that for every $s \in S(k)$, the image of the restriction $f_s \colon \mathcal{X}_s \to Y$ is not contained in the support of $\Delta_Y$.
Furthermore, assume that there is a dense set of points $s$ in $S(k)$ such that $f_s \colon \mathcal{X}_s \to (Y,\Delta_Y)$ is an orbifold morphism.
Then, $f \colon \mathcal{X} \to (Y, \Delta_Y)$ is an orbifold morphism and the fiberwise morphisms $f_s \colon \mathcal{X}_s \to (Y,\Delta_Y)$ are orbifold for all $s \in S(k)$.
\end{theorem}

\begin{proof}
First, note that by assumption the image of $f$ is not contained in $\supp(\Delta_Y)$. Thus, we only need to check that the multiplicities of the preimage divisors are correct. For this, let $E \subseteq \supp(\Delta_Y)$ be a prime divisor and let $m$ be the multiplicity of $E$ in $\Delta_Y$. Let $D$ be an irreducible component of the pullback divisor $f^*E$.

Observe that $\pi|_D \colon D \to S$ is dominant. Otherwise, the generic point $\eta$ of $D$ maps to a non-generic point $\pi(\eta) \in S$. As $\pi^{-1}(\pi(\eta))$ is irreducible by assumption, and since we have $D \subseteq \overline{\pi^{-1}(\pi(\eta))}$ but $\overline{\pi^{-1}(\pi(\eta))} \neq \mathcal{X}$ by dominance of $\pi$, we have $D = \overline{\pi^{-1}(\pi(\eta))}$. This implies that there is a closed point $s \in S(k)$ which is a specialization of $\pi(\eta)$ such that $D$ contains the fiber $\mathcal{X}_s$. But this is in contradiction to the assumption that no fiber of $\pi$ is mapped into the support of $\Delta_Y$.

Let $\mathcal{L}$ be the line bundle on $Y$ which corresponds to the divisor class $E$. Let $t \in \mathcal{L}(Y)$ be the global section whose vanishing divisor equals $E$. We pull back $t$ and $\mathcal{L}$ along $f$ and obtain a line bundle $f^*\mathcal{L}$ on $\mathcal{X}$ with a global section $f^*t$. This global section vanishes along $D$. We must show that it does so with multiplicity at least $m$.

If $m = \infty$, then we know that $D \cap \mathcal{X}_s$ must be empty for every $s \in S(k)$ such that $f_s \colon \mathcal{X}_s \to (Y, \Delta_Y)$ is an orbifold morphism. This contradicts the observation that $D \to S$ is dominant, showing that $D$ cannot exist, i.e. $f^*E = 0$.

If $m < \infty$, let $D_m$ be the $m$-th infinitesimal neighbourhood of $D$ in $\mathcal{X}$, i.e., the closed subscheme cut out by the ideal sheaf $\mathcal{I}^m$, where $\mathcal{I} \subseteq \mathcal{O}_{\mathcal{X}}$ denotes the ideal sheaf of $D$.
Note that $D_m$ is an irreducible closed subscheme of $\mathcal{X}$.
As $\mathcal{X}$ is locally factorial, the divisor $D \subseteq \mathcal{X}$ is locally principal.
Hence the scheme $D_m$ is locally cut out by a single equation and thus, has no embedded points \cite[Tag~031T]{stacks-project}.
By what we observed before, the projection $\pi|_{D_m} \colon D_m \to S$ is dominant.

We claim that $f^\ast t \in f^\ast \mathcal{L}(X)$ vanishes along $D$ with multiplicity at least $m$ if and only if $f^\ast t$ is contained in $\mathcal{I}^m \mathcal{L}$.
This claim is clearly local on $\mathcal{X}$, and thus to prove it, we may assume that $\mathcal{X} = \Spec R$ is affine, that the line bundle $f^* \mathcal{L}$ is trivial and that $\mathcal{I}$ corresponds to a principal ideal $\mathfrak{p} \subseteq R$ (which is automatically prime since $D$ is integral).
The section $f^* t$ vanishes along $D$ with multiplicity at least $m$ if and only if $f^*t \in \mathfrak{p}^m R_\mathfrak{p}$, so that our claim is equivalent to the equality of ideals $R \cap \mathfrak{p}^m R_\mathfrak{p} = \mathfrak{p}^m$.
This equality of ideals holds since $\mathfrak{p}$ is principal; hence the claim is true.
Thus, we must show that $f^*t$ lies in $\mathcal{I}^m \mathcal{L}$ or equivalently, that it vanishes when restricted to $D_m$.

To show that $f^*t$ vanishes when restricted to $D_m$, we first note that if the fiber $D_s := D \times_S \kappa(s)$ is reduced for some $s \in S(k)$, the $m$-th infinitesimal neighborhood of $D \times_S \kappa(s)$ in $\mathcal{X}_s$ is exactly $D_m \times_S \kappa(s)$.
Since the generic fiber of $D \to S$ is geometrically reduced (as we are in characteristic zero), the set of $s \in S(k)$ such that $D_s$ is reduced contains a dense open \cite[Tag~0578]{stacks-project}.

By assumption, there is a dense subset $\Sigma \subseteq S(k)$ such that, for every $s \in \Sigma$, the morphism $f_s \colon \mathcal{X}_s \to (Y,\Delta_Y)$ is orbifold. For these $s$, the section $f^*t$ vanishes in the $m$-th infinitesimal neighborhood of $(D \times_S \kappa(s))_{\mathrm{red}}$. By the previous paragraph, we may assume that for all $s \in \Sigma$, the fiber $D_s$ is reduced. Therefore, for every $s \in \Sigma$, the section $f^*t$ vanishes identically on the fiber of $D_m \to S$ over $s$. It then follows from Proposition \ref{proposition:vanishing_of_t} that $(f^*t)|_{D_m}$ vanishes. This proves that $\mathcal{X} \to (Y,\Delta_Y)$ is an orbifold morphism. 

To show that all fiberwise morphisms are orbifold, let $s \in S(k)$ be a closed point and consider the morphism $\mathcal{X}_s \subset \mathcal{X} \to (Y,\Delta_Y)$. This is a composition of orbifold morphisms. As its image is not contained in $\supp \Delta_Y$ by assumption, it is an orbifold morphism.
\end{proof}

\subsection{The Hom-scheme of orbifold maps}

For $X$ and $Y$ projective schemes over a field $k$ of characteristic zero, we let $\underline{\Hom}_k(X,Y)$ be the scheme representing the functor 
\[
\mathrm{Sch}/k^{op}\to \mathrm{Sets}, \quad S \mapsto \Hom_S(X_S,Y_S).
\]  
Fix ample line bundles on $X$ and $Y$, and  hence on $X\times Y$. Then, for every polynomial $P \in \mathbb{Q}[t]$, let $\underline{\Hom}_k^P(X,Y)$ be the subscheme of $\underline{\Hom}_k(X,Y)$ parametrizing morphisms $f \colon X \to Y$ with Hilbert polynomial $P$ (with respect to the fixed ample line bundle on $X\times Y$). Note that $\underline{\Hom}_k^P(X,Y)$ is a quasi-projective scheme over $k$ and that $\underline{\Hom}_k(X,Y) = \sqcup_{P \in \mathbb{Q}[t]} \underline{\Hom}_k^P(X,Y)$.

Let $\Delta_Y$ be an orbifold divisor on $Y$. Let $\Hom(X,(Y,\Delta_Y))$ be the subset of $\Hom_k(X,Y)$ given by orbifold morphisms $X \to (Y,\Delta_Y)$. Note that $\Hom(X,(Y,\Delta_Y)) \subseteq \Hom_k(X,Y) \setminus \Hom_k(X, \supp \Delta_Y)$.  

\begin{lemma}\label{closed_hom}
Let $S$ be a noetherian scheme. Let $X \to S$ be a flat projective morphism, and let $Y \to S$ be a quasi-projective morphism. Let $Z \subseteq X$ be a closed subscheme, flat over $S$. Then there is a natural closed immersion $\underline{\Hom}_S(Y,Z) \to \underline{\Hom}_S(Y,X)$.
\end{lemma}
\begin{proof}
See \cite[Variant 4.c]{FGA4} for the existence of the representing schemes mentioned in this proof. Let $T$ be any $S$-scheme. Then an $S$-morphism $Y \times_S T \to X$ factors over $Z$ if and only if its graph $\Gamma \subseteq Y \times_S T \times_S X$ is contained in $Y \times_S T \times_S Z$. This shows that $\underline{\Hom}_S(Y,Z) = \underline{\Hom}_S(Y,X) \times_{\mathrm{Hilb}_S(Y \times_S X)} \mathrm{Hilb}_S(Y \times_S Z)$. Thus, it suffices to show that $\mathrm{Hilb}_S(Y \times_S Z) \to \mathrm{Hilb}_S(Y \times_S X)$ is a closed immersion. This is shown in \cite[Tag~0DPF]{stacks-project}.
\end{proof}

\begin{lemma}\label{lemma:immer_plus_orb_is_orb}
Let $f \colon Z \to X$ be an immersion of normal varieties over a field $k$ of characteristic zero, let $(Y,\Delta)$ be a locally factorial orbifold, and let $g \colon X \to (Y,\Delta_Y)$ be an orbifold morphism. If the image of the composed map $g \circ f$ is not contained in $\supp \Delta_Y$, then $g \circ f$ is an orbifold morphism $Z \to (Y,\Delta_Y)$.
\end{lemma}
\begin{proof}
As every immersion can be factored into an open and a closed immersion, it suffices to treat these cases separately.

The case of an open immersion is clear, as the coefficients of a Weil divisor (like $g^* \Delta_Y$) do not change when restricting to a dense open. 

For the case of a closed immersion, let $\mathrm{Div}_Z(X) \subseteq \mathrm{Div}(X)$ denote the subgroup of those divisors on $X$ whose support does not contain $Z$. Then there is a well-defined restriction map $\mathrm{Div}_Z(X) \to \mathrm{Div}(Z)$ sending effective divisors to effective divisors. In particular, this restriction map does not decrease the coefficients of an effective Weil divisor. As $g$ does not factor over $\supp \Delta_Y$, the pullback of every irreducible component of $\Delta_Y$ along $g$ is actually contained in the subgroup $\mathrm{Div}_Z(X)$. Moreover, the pullback along $g \circ f$ (whenever it is defined) is the pullback along $g$ followed by the restriction map $\mathrm{Div}_Z(X) \to \mathrm{Div}(X)$. Combining these statements shows that the pullback of every irreducible component of $\supp \Delta_Y$ along $g \circ f$ has sufficiently high multiplicity, as desired. 
\end{proof}

\begin{lemma}\label{lemma:closure_preimage_fieldextension}
Let $X$ be a scheme over a field $k$ and let $S \subseteq X$ be a subset with closure $Z \subseteq X$. Let $k \subseteq l$ be a field extension and let $S_l$ be the preimage of $S$ under the map $X_l \to X$. Then the base change $Z_l$ is the closure of $S_l$ in $X_l$.
\end{lemma}
\begin{proof}
As the map $\Spec l \to \Spec k$ is universally open \cite[Tag~0383]{stacks-project}, the map $X_l \to X$ is open. Now recall that for an open continuous map of topological spaces, taking closures commutes with taking preimages. Hence the closure of $S_l$ in $X_l$ is the preimage of $Z$ under $X_l \to X$. Lastly, note that the underlying set of the schematic base change $Z_l$ agrees with the set-theoretic preimage of $Z$ in $X_l$, as $Z \subseteq X$ is closed.
\end{proof}

\begin{theorem} \label{thm:algebraicity_of_moduli}
Let $k$ be a field of characteristic zero. Let $X$ be a normal quasi-projective variety over $k$ and let $X \subseteq \overline{X}$ be a normal projective compactification with divisorial boundary $\Delta_X$. Let $(Y, \Delta_Y)$ be a locally factorial orbifold. Then there is a closed subscheme  
\[ H \subseteq \underline{\Hom}_k(\overline{X},Y) \setminus \underline{\Hom}_k(\overline{X},\supp \Delta_Y) \]
satisfying the following properties:
\begin{itemize}
\item Formation of $H$ commutes with algebraic extensions of the base field $k$.
\item For every algebraic field extension $k \subseteq l$, the set $H(l)$ is the set of orbifold morphisms $(\overline{X}, \Delta_X)_l \to (Y, \Delta_Y)_l$.
\item For every normal variety $T$ over $k$, the set $H(T)$ is the set of orbifold morphisms $(\overline{X} \times T, \Delta_X \times T) \to (Y, \Delta_Y)$ such that, for every closed point $t \in T$, the induced morphism of varieties $\overline{X} \times \{t\} \to Y$ does not factor over $\supp \Delta_Y$.
\item For every irreducible component of $H$ with normalization $H'$, the evaluation morphism $\ev \colon X \times H' \to (Y, \Delta_Y)$ is an orbifold morphism.
\end{itemize}
\end{theorem}
\begin{proof}
Consider the set
\[ S := \left\{ x \in \underline{\Hom}_k(\overline{X},Y)~\bigg{|} ~\begin{aligned}x~ \text{is a closed point and represents an} \\ \text{orbifold morphism}~(\overline{X}, \Delta_X)_{\overline{k}} \to (Y, \Delta_Y)_{\overline{k}} \end{aligned} \right\} \]
and let $H$ be its closure in $\underline{\Hom}_k(\overline{X},Y) \setminus \underline{\Hom}_k(\overline{X},\supp \Delta_Y)$. Note that the latter object is a scheme by Lemma \ref{closed_hom}. We endow $H$ with the reduced scheme structure. Let $k \subseteq l$ be an algebraic field extension and consider the subset
\[ S' := \left\{ x \in \underline{\Hom}_l(\overline{X}_l,Y_l)~\bigg{|}~\begin{aligned} x~\text{is a closed point and represents an} \\ \text{orbifold morphism}~(\overline{X}, \Delta_X)_{\overline{k}} \to (Y, \Delta_Y)_{\overline{k}}  \end{aligned} \right\}. \]
Observe that $S'$ is the preimage of $S$ under the natural map $\underline{\Hom}_l(X_l, Y_l) \to \underline{\Hom}_k(X,Y)$ and recall that forming the $\Hom$-scheme commutes with arbitrary extension of the base field. Thus, by Lemma \ref{lemma:closure_preimage_fieldextension}, the base change $H_l$ is the closure of $S'$ in $\underline{\Hom}_l(\overline{X}_l,Y_l) \setminus \underline{\Hom}_l(\overline{X}_l,(\supp \Delta_Y)_l)$. This shows that our construction of $H$ commutes with extending the base field from $k$ to $l$, thus proving the first bullet point.

If $\overline{X} \to Y$ is a morphism of varieties defined over $k$, then the base change $\overline{X}_{\overline{k}} \to Y_{\overline{k}}$ is an orbifold morphism $(\overline{X}, \Delta_X)_{\overline{k}} \to (Y, \Delta_Y)_{\overline{k}}$ if and only if the original map is an orbifold morphism $(\overline{X}, \Delta_X) \to (Y, \Delta_Y)$ over $k$. Thus, to prove the remaining statements, we may assume that $k$ is algebraically closed.

We next show that the evaluation morphism is an orbifold morphism. For this, let $H'$ be the normalization of an irreducible component of $H$. Let $X^o \subseteq X$ and $H'^o \subseteq H'$ denote the smooth loci, which are big opens by the normality of $X$ and $H'$. The variety $X^o \times H'^o$ is then smooth and the projection morphism $X^o \times H'^o \to H'^o$ has geometrically integral fibers. Moreover, by construction, there is a dense set of points $h \in H'^o(k)$ such that the induced morphism $X^o \to (Y, \Delta_Y)$ is an orbifold morphism. Thus, by Theorem \ref{thm:families_of_orbifold_maps}, the map $X^o \times H'^o \to (Y, \Delta_Y)$ is an orbifold morphism. Since the complement of $X^o \times H'^o$ in $X \times H'$ has codimension at least two (by normality of $X$ and $H'$), it follows that $X \times H'$ and $X^o \times H'^o$ have the same Weil divisors and hence that $X \times H' \to (Y, \Delta_Y)$ is orbifold, as desired.

Next, we show that the set $H(k)$ is the set of orbifold morphisms $(\overline{X}, \Delta_X) \to (Y, \Delta_Y)$. By construction, we have that every orbifold morphism $(\overline{X}, \Delta_X) \to (Y, \Delta_Y)$ is an element of $H(k)$. Thus, we have to show that every element of $H(k)$ defines an orbifold morphism $(\overline{X}, \Delta_X) \to (Y, \Delta_Y)$. To do so, first observe that an orbifold morphism $(\overline{X}, \Delta_X) \to (Y, \Delta_Y)$ is a morphism of varieties $\overline{X} \to Y$ such that the induced map $X \to (Y, \Delta_Y)$ is orbifold and vice versa. Consequently, it suffices to show that every element $h$ of $H(k)$ induces an orbifold morphism $X \to (Y, \Delta_Y)$. To do so, let $H'$ denote the normalization of an irreducible component of $H$ containing $h$. Then the map represented by $h$ is the composition of the inclusion $X \times \{h\} \subseteq X \times H'$ and the evaluation map $X \times H' \to (Y, \Delta_Y)$. This composition is an orbifold map by Lemma \ref{lemma:immer_plus_orb_is_orb}.

It remains to show the third bullet point. To do so, let $T$ be a normal variety over $k$. Then an element of $H(T)$ defines a morphism of varieties $\overline{X} \times T \to Y$. Moreover, for every $t$ in $T(k)$, the induced map $X \times \{t\} \to (Y, \Delta_Y)$ is an orbifold morphism by the preceding paragraph. In particular, it does not factor over $\supp \Delta_Y$. Let $X^o$ and $T^o$ denote the smooth loci of $X$ and $T$, which by our normality assumptions are big opens. Then, applying Theorem \ref{thm:families_of_orbifold_maps} to the projection $X^o \times T^o \to T^o$, we see that the map $X^o \times T^o \to (Y, \Delta_Y)$ is an orbifold morphism. Again using that $X^o \times T^o \subseteq X \times T$ is a big open, we see that $X \times T \to (Y, \Delta_Y)$ and hence $(\overline{X} \times T, \Delta_X \times T) \to (Y, \Delta_Y)$ is an orbifold morphism, as well. Conversely, if $(\overline{X} \times T, \Delta_X \times T) \to (Y, \Delta_Y)$ is an orbifold morphism such that for every $t \in T(k)$, the induced morphism $\overline{X} \times \{t\} \to Y$ does not factor over $\supp \Delta_Y$, it follows from Lemma \ref{lemma:immer_plus_orb_is_orb} that then the induced map $X \times \{t\} \to (Y, \Delta_Y)$ is orbifold for every $t \in T(k)$. As the $k$-rational points of $T$ are dense in $T$, this shows that the induced morphism $T \to \underline{\Hom}(\overline{X}, Y) \setminus \underline{\Hom}(\overline{X}, \supp \Delta_Y)$ factors set-theoretically over $H$. As $T$ is reduced, we get a morphism of schemes $T \to H$, i.e. an element of $H(T)$. This finishes the proof.
\end{proof}

We note that the fact that the subset of orbifold maps in the moduli space of all maps is locally closed was also proven by Kebekus--Pereira--Smeets; see \cite[Claim~10.1]{Smeets}.

In the remainder of the paper, we will write $\underline{\Hom}((\overline{X}, \Delta_X), (Y, \Delta_Y))$ for the scheme $H$ constructed in the above theorem. In the special case that $X = C$ is a smooth quasi-projective curve, we will also slightly abuse notation and write $\underline{\Hom}(C, (Y, \Delta_Y))$ instead. We explicitly note the following corollary.

\begin{corollary} \label{cor:hom_curve_to_orbifold}
Let $(X, \Delta_X)$ be a projective locally factorial orbifold and let $C$ be a smooth quasi-projective curve with smooth compactification $C \subseteq \overline{C}$. Then the scheme $\underline{\Hom}(C, (X, \Delta_X))$ is closed in $$\underline{\Hom}(\overline{C}, X) \setminus \underline{\Hom}(\overline{C}, \supp \Delta_X)$$ and for every irreducible component of $\underline{\Hom}(C, (X, \Delta_X))$ (endowed with its reduced subscheme structure) with normalization $H'$, the evaluation map $\ev \colon C \times H' \to (X, \Delta_X)$ is an orbifold morphism.
\end{corollary} 
\begin{proof}
Noting that every orbifold morphism $C \to (X, \Delta_X)$ extends to a morphism of varieties $\overline{C} \to X$, this is just a special case of Theorem \ref{thm:algebraicity_of_moduli}.
\end{proof}

This immediately implies the following (presumably well-known) result.

\begin{corollary}
Let $\overline{C}$ be a smooth projective curve and let $X$ be a locally factorial projective variety over an algebraically closed field $k$ of characteristic zero. Let $C \subset \overline{C}$ be a dense open and $U \subset X$ be a dense open with complement $D$. Then $\Hom(C,U) \subset \underline{\Hom}(\overline{C},X)(k) \setminus \underline{\Hom}(\overline{C},D)(k)$ is closed.
\end{corollary}

\section{The dimension of the moduli space of orbifold maps}\label{section:dimension_of_moduli}
In this section, let $k$ be an algebraically closed field of characteristic zero.
We show that, under suitable assumptions, the moduli space of non-constant orbifold maps from a fixed curve to an orbifold pair of general type $(X,\Delta)$ is at most one-dimensional (see Corollary \ref{corollary:boundedness}). 

Let $(X,\Delta)$ be a smooth proper orbifold. Recall that $(X,\Delta)$ is of general type if $K_X + \Delta$ is a big $\mathbb{Q}$-divisor.  To prove Corollary \ref{corollary:boundedness} we invoke the following recent finiteness result for dominant maps \cite[Theorem~1.1]{BJ}.

\begin{theorem}[Kobayashi--Ochiai for orbifold pairs]\label{thm:ko} 
Let $V$ be a normal integral variety and let $(X,\Delta)$ be a smooth proper orbifold of general type. Then, the set of dominant morphisms $V \to (X,\Delta)$ is finite. 
\end{theorem}

\subsection{Bend-and-break and orbifold Kobayashi--Ochiai}
 Given a smooth quasi-projective curve $C$ and a smooth projective orbifold $(X,\Delta)$,  we let  $\underline{\Hom}_k^{nc}(C,(X,\Delta))$ be the subscheme of $\underline{\Hom}_k(C,(X,\Delta))$ parametrizing non-constant morphisms from $C$ to $(X,\Delta)$.  We will use Theorem \ref{thm:ko} and Mori's bend-and-break to prove the following structure result for such moduli spaces of orbifold maps.
 
\begin{lemma} \label{lemma:ko_app}  
Let $(X,\Delta)$ be a smooth projective orbifold of general type and let $C$ be a smooth quasi-projective curve. Let $H \subset \underline{\Hom}_k^{nc}(C,(X,\Delta))$ be a locally closed subscheme of dimension at least $\dim X$. Assume $H$ is an integral finite type scheme over $k$. Then, the image of $C \times H \to X$ is uniruled.  
\end{lemma}

\begin{proof}
Replacing $H$ by a dense open if necessary, we may assume that $H$ is smooth.

Let $\Sigma := \{c \in C(k) \ | \ \ev_c(H) \subset \Delta\}$. Note that $\Sigma$ is a finite subset of $C$. It follows from Corollary \ref{cor:hom_curve_to_orbifold} that the universal evaluation map $C \times H \to X$ defines an orbifold morphism $C \times H \to (X,\Delta)$. In particular, for every $c \in C(k) \setminus \Sigma$, the morphism $\ev_c \colon H \to (X,\Delta)$ is orbifold. 
 
By Theorem \ref{thm:ko}, as $H$ is a variety, the set of $c$ in $C(k) \setminus \Sigma$ with $\ev_c$ dominant is finite. Indeed, assume for a contradiction that there is a sequence $c_1,c_2,\ldots$ of pairwise distinct points of $C \setminus \Sigma$ such that $\ev_{c_1}$, $\ev_{c_2}, \ldots$ are dominant. Then, by Theorem \ref{thm:ko}, replacing $c_1, c_2, \ldots$ by a subsequence if necessary, we must have that $\ev_{c_1} = \ev_{c_2} = \ldots$. This implies that, for all $f$ in $H$, we have $f(c_1) = f(c_2) = f(c_3) = \ldots$, i.e., $f$ is constant. This contradicts our assumption that $H$ lies in the moduli space of non-constant maps from $C$ to $X$. We conclude that there is a dense open $U \subset C(k) \setminus \Sigma$ such that, for all $c$ in $U$, the morphism $\ev_c \colon H \to X$ is non-dominant.

We now adapt part of the proof of \cite[Lemma~2.2.1]{GuerraPirola}. Note that, for all $c$ in $U$, the closure of the image of $\ev_c$ is of dimension at most $\dim X-1$. Thus, since $H$ has dimension at least $\dim X$, for every $c$ in $U$, there is a dense open $V_c \subset \ev_c(H)$ such that, for every $x$ in $V_c$, the fiber of $\ev_c \colon H \to X$ over $x$ is positive-dimensional. 
In particular, for every $x$ in $V_c$, there is a curve $D_x \subset H$ that is contracted to $x$ along $\ev_c$. Consequently, the scheme $\underline{\Hom}((C,c),(X,x)) \cap H$ is positive-dimensional (as it contains the curve $D_x$).

Let $\overline{C}$ be the smooth projective model for $C$. For every $f$ in $H$ given by a morphism $f \colon C \to X$, we let $\overline{f}\colon \overline{C}\to X$ denote the unique extension to $\overline{C}$.
Now, let $L$ be an ample line bundle on $X$ and note that, since $H$ is of finite type over $k$, there is a constant $\alpha$ (depending only on $L$ and $H$) such that, for every $f \in H$, the degree of $\overline{f}^\ast L$ on $\overline{C}$ is bounded by $\alpha$. In particular, for every $c$ in $U$ and $x$ in $V_c$, the moduli scheme $\underline{\Hom}^{\leq \alpha}((\overline{C},c),(X,x))$ of morphisms $f \colon \overline{C} \to X$ with $f(c) = x$ and $\deg \overline{f}^\ast L \leq \alpha$ is positive-dimensional, so that there is a rational curve of degree at most $2 \alpha$ in $X$ passing through $x$ (see \cite[Proposition~3.5]{DebarreBook}). 
We conclude that, for every $x$ in $\cup_{c\in U} V_c$, there is a rational curve $\mathbb{P}^1\to X$ of degree at most $2\alpha$ passing through $x$.  

Let $Z$ be the closure of the image of $C \times H \to X$ and note that the closure of $\cup_{c\in U} V_c$ equals $Z$. Then, the morphism
\[
\mathbb{P}^1_k \times \underline{\Hom}^{\leq 2\alpha}(\mathbb{P}^1,Z) \to Z
\]
is dominant (as its image contains $\cup_{c\in U} V_c$), so that $Z$ is uniruled. This concludes the proof.
\end{proof}
 
\begin{remark}\label{remark:assumption}
Let $(X,\Delta)$ be a smooth proper orbifold and let $Z \subseteq X$ be a closed subset. The following conditions are equivalent:
\begin{enumerate}
\item For every smooth quasi-projective curve $C^0$ and every orbifold morphism $C^0 \to (X,\Delta)$ not factoring over $Z$, the curve $C^0$ is of log-general type.
\item For every smooth proper orbifold curve $(C,\Delta_C)$ and every morphism $f \colon (C,\Delta_C)\to (X,\Delta)$ with $f(C)\not \subset Z$, we have that $(C,\Delta_C)$ is of general type.
\end{enumerate}  Indeed, to show that $(2) \implies (1)$, let $C$ be the smooth projective model of $C^0$ and let $\Delta_C$ be the divisor $C \setminus C^0$ (where we give each point multiplicity $\infty$). To show that $(1)\implies (2)$, we invoke the following fact:  for every smooth proper orbifold curve $(C,\Delta_C)$ which is not of general type, there  is a dominant orbifold morphism $\mathbb{G}_m \to (C,\Delta_C)$ or a dominant orbifold morphism $E\to (C,\Delta_C)$ with $E$ an elliptic curve.  
\end{remark}
 
\begin{lemma}\label{lemma:two_dim}
Let $(X,\Delta)$ be a smooth projective orbifold and let $Z \subset X$ be a closed subset. Assume that, for every smooth quasi-projective curve $D$ and orbifold morphism $f \colon D \to (X,\Delta)$ with $f(D) \not\subset Z$, the curve $D$ is of log-general type. Let $C$ be a smooth quasi-projective curve. If $H$ is a positive-dimensional integral locally closed subscheme of $\underline{\Hom}_k^{nc}(C,(X,\Delta))\setminus \underline{\Hom}_k(C,Z)$, then the image of the universal evaluation map $C \times H \to X$ is at least two-dimensional.  
\end{lemma}
\begin{proof}
Clearly, the image of $C \times H \to X$ is positive-dimensional.    Let $D \subset X$ be the scheme-theoretic image of $\ev \colon C \times H \to X$ and note that $D$ is an integral closed subscheme which is not contained in $Z$ (as $H$ is not contained in $\underline{\Hom}_k(C,Z)$). Assume for a contradiction that $D$ is one-dimensional.  

Let $D' \to D$ be the normalization of the integral curve $D$. Let $\Delta_{D'}$ be the orbifold divisor on $D'$ induced by $\Delta$, i.e., $(D',\Delta_{D'})\to (X,\Delta)$ is a morphism of orbifolds and $\Delta_{D'}$ is minimal with this property.  Since the image of $D'\to X$ is not contained in $Z$, it follows from our assumption (see Remark \ref{remark:assumption}) that the smooth proper orbifold $(D',\Delta_{D'})$ is of general type. Now, for every $f$ in $H$, by the universal property of normalizations, the morphism $f \colon C \to X$ factors over $D'\to X$.  Moreover, by construction of $\Delta_{D'}$,  the morphism $C\to D'$ is in fact a dominant orbifold morphism $C \to (D',\Delta_{D'})$. However, by Campana's De Franchis theorem \cite[\S 3]{Ca05} (or Theorem \ref{thm:ko}), the set of dominant orbifold morphisms $C \to (D', \Delta_{D'})$ is finite. This implies that $H$ is finite contradicting the positive-dimensionality of $H$.  

We conclude that $D$ is at least two-dimensional, as required.
\end{proof}
 
\begin{theorem}\label{thm:culmination}
Let $(X,\Delta)$ be a smooth projective orbifold surface of general type with $X$ non-uniruled. Let $Z \subseteq X$ be a closed subvariety such that, for every smooth quasi-projective curve $D$ and every morphism $f \colon D \to (X,\Delta)$ with $f(D) \not\subset Z$, the curve $D$ is of log-general type. Let $C$ be a smooth quasi-projective curve. Then, the moduli scheme $\underline{\Hom}^{nc}(C,(X,\Delta)) \setminus \underline{\Hom}(C,Z)$ is of dimension at most one.
\end{theorem}
\begin{proof}
We argue by contradiction. Assume $H \subset \underline{\Hom}^{nc}(C,(X,\Delta)) \setminus \underline{\Hom}(C,Z)$ is an irreducible component of dimension at least two. First, we note that the image of $C \times H \to X$ is of dimension at least two (Lemma \ref{lemma:two_dim}), and thus the morphism $C \times H \to X$ is dominant. However, by Lemma \ref{lemma:ko_app}, the image of $C \times H \to X$ is uniruled. We conclude that $X$ is uniruled, contradicting our assumption.  
\end{proof}

\subsection{Pseudo-bounded orbifold pairs}
We say that a locally factorial projective orbifold $(X,\Delta)$ is \emph{$1$-bounded modulo $Z$} if, for every ample line bundle $L$ on $X$ and every smooth projective connected curve $C$ over $k$, there is a constant $\alpha_{X,\Delta,L,Z,C}$ such that,  for every orbifold morphism $f \colon C \to (X,\Delta)$ with $f(C)\not\subset \Delta\cup Z$, the inequality
\[
\deg_C f^\ast L \leq \alpha_{X,\Delta,L,Z,C}  
\] 
holds.  
We say that a locally factorial projective orbifold $(X,\Delta)$ is \emph{pseudo-$1$-bounded} if there is a proper closed subset $Z \subsetneq X$   such that $(X,\Delta)$ is $1$-bounded modulo $Z$. 

\begin{corollary}\label{corollary:boundedness}
Let $(X,\Delta)$ be a smooth projective orbifold surface of general type with $X$ non-uniruled, and let $Z \subsetneq X$ be a proper closed subvariety such that $(X,\Delta)$ is $1$-bounded modulo $Z$.
Then, the following statements hold.
\begin{enumerate}
\item The moduli scheme $\underline{\Hom}^{nc}(C,(X,\Delta))\setminus \underline{\Hom}(C,Z)$ is a scheme of finite type over $k$ of dimension at most one.
\item For almost every $c$ in $C$, the evaluation map $\ev_c \colon \underline{\Hom}^{nc}(C,(X,\Delta))\setminus \underline{\Hom}(C,Z) \to X$ is quasi-finite.
\end{enumerate}
\end{corollary}
\begin{proof}
The scheme $\underline{\Hom}^{nc}(C,(X,\Delta))\setminus \underline{\Hom}(C,Z)$ is of finite type over $k$ by the assumption that $(X,\Delta)$ is $1$-bounded modulo $Z$. Moreover, this assumption implies that, for every smooth quasi-projective curve $C$ over $k$ and every $f \colon C\to (X,\Delta)$ with $f(C) \not \subset Z$, the curve $C$ is of log-general type. (Here we use that a curve which is not of log-general type has endomorphisms of unbounded degree.) Therefore, the first statement follows directly from Theorem \ref{thm:culmination}. 

The second statement follows easily from the first statement, as we show now. Namely, let $H_1,\ldots, H_r \subset \underline{\Hom}^{nc}(C,(X,\Delta))\setminus \underline{\Hom}(C,Z)$ be the positive-dimensional irreducible components of $\underline{\Hom}^{nc}(C,(X,\Delta))\setminus \underline{\Hom}(C,Z)$. Note that the set $\Sigma$ of $c$ in $C(k)$ such that there is an $1\leq i \leq r$ with $\ev_c \colon H_i\to X$ constant is finite. In particular, for every $c$ in $C(k)\setminus \Sigma$, the morphism $\ev_c \colon H_i \to X$ is non-constant, and thus quasi-finite (as $H_i$ is one-dimensional). It readily follows that, for every $c$ in $C\setminus \Sigma$, the morphism $\ev_c \colon \underline{\Hom}^{nc}(C,(X,\Delta)) \setminus \underline{\Hom}(C,Z) \to X$ is quasi-finite. 
\end{proof}

Corollary \ref{corollary:boundedness} implies that $\underline{\Hom}^{nc}(C,(X,\Delta))\setminus \underline{\Hom}(C,Z)$ is a hyperbolic quasi-projective scheme of dimension at most one. In particular, it satisfies Lang--Vojta's conjectures and has only finitely many integral points (on any model over the integers). However, we require the finiteness of rational points on this space, so we wish to establish that   the smooth projective model of every one-dimensional  component of the moduli space $\underline{\Hom}^{nc}(C,(X,\Delta)) \setminus \underline{\Hom}(C,Z)$ has genus at least two. This is however false without any additional assumptions (Remark \ref{remark:assumption2}). In the next section we prove the desired property, assuming the variety $X$ underlying the orbifold $(X,\Delta)$ is a surface of Kodaira dimension one whose elliptic fibration is non-isotrivial.

\section{Kodaira dimension one and rational points on moduli spaces of orbifold maps}\label{section:kodaira_dim_one}
Using our results above, we now prove the following structure result for the moduli space of orbifold maps, assuming that the variety underlying the orbifold is a non-isotrivial Kodaira dimension one surface.  

\begin{theorem}\label{theorem:Kodaira_dimension_one}  
Let $(X,\Delta)$ be a smooth projective orbifold of general type with $X$ a Kodaira dimension one surface whose elliptic fibration is non-isotrivial. Let $Z \subset X$ be a proper closed subset such that $(X,\Delta)$ is $1$-bounded modulo $Z$. Let $C$ be a smooth quasi-projective curve and let $H \subset \underline{\Hom}_k^{nc}(C,(X,\Delta)) \setminus \underline{\Hom}(C,Z)$ be a positive-dimensional irreducible component. Then $H$ is birational to a smooth projective curve of genus at least two.
\end{theorem}
\begin{proof}
Since $X$ is non-uniruled, it follows from Corollary \ref{corollary:boundedness}.(1) that $H$ is one-dimensional. Moreover, the evaluation map $C \times H \to X$ is dominant (Lemma \ref{lemma:two_dim}). Let $\overline{H}$ be the smooth projective model of $H$, and let $g$ be its genus. To prove the corollary, it suffices to show that $g \geq 2$. First, since $X$ is not uniruled, we see that $g \geq 1$. We now argue by contradiction and assume that $g=1$.  

Let $X \to B$ be the non-isotrivial elliptic fibration on $X$ (induced by the pluricanonical system of $X$). Let $C\subseteq \overline{C}$ be the smooth compactification  of $C$.  Since $C \times \overline{H}$ dominates $X$ (see Remark \ref{remark:kloosterman})   and $\overline{H}$ has genus one, we have that $\overline{C} \times \overline{H}$ is of Kodaira dimension one. In particular, the Iitaka fibration of $\overline{C}\times \overline{H}$ is the projection onto $\overline{C}$. Therefore, by the canonicity of the Iitaka fibration, there is a commutative diagram
\[
\xymatrix{
\overline{C} \times \overline{H} \ar[d] \ar@{-->}[rrr]^{\textrm{rational map}} &&& X \ar[d] \\ 
\overline{C} \ar[rrr] &&& B.
}
\]
It follows that almost all fibers of $X \to B$ are dominated by the elliptic curve $\overline{H}$. In particular, it follows that almost all fibers of $X \to B$ are isogenous to each other, so that $X \to B$ is isotrivial. This contradiction completes the proof.
\end{proof}

\begin{remark}\label{remark:kloosterman}
We note that, with notation as in the above proof, the analysis of the case that $C\times \overline{H}$ dominates $X$ can not be omitted. Indeed, there exist smooth projective surfaces $X$ of Kodaira dimension one which are dominated by a product of two curves of genus at least two. One can construct such surfaces as follows (using some of the notions studied in \cite{Heijne}):
For every even integer $n \geq 12$, consider the smooth projective surface $X$ defined by the affine equation
\[
1+x t^n + x^3 + y^2=0
\]
in $\mathbb{A}^2_{x,y} \times \mathbb{A}^1_t$. Then $X$ is a Kodaira dimension one surface with non-isotrivial elliptic fibration (given by the projection onto $t$). This surface is dominated by a Fermat surface $X^m + Y^m = Z^m + W^m$ for some large integer $m$. Let $C$ be the smooth affine curve defined by $x^m + y^m = 1$ and note that $C \times C$ dominates $X$.   
\end{remark}

\begin{remark}\label{remark:assumption2}
Theorem \ref{theorem:Kodaira_dimension_one} is false without the assumption that $X$ is of Kodaira dimension one and non-isotrivial. Consider, for example, the smooth projective orbifold of general type $(X,\Delta) := (E \times E,  E \times \frac{1}{2} \{0\} + \frac{1}{2} \{0\} \times E)$ where $E$ is an elliptic curve and $0 \in E$ is its origin. Then, if $C$ is any smooth quasi-projective curve dominating $(E, \frac{1}{2} \{0\})$, the moduli space $\underline{\Hom}(C, (X, \Delta))$ contains a copy of $E \setminus \{0\}$.
\end{remark}

\begin{corollary}\label{cor:finiteness_for_orbifold_surface1}
Let $X$ be a Kodaira dimension one smooth projective surface over a finitely generated field $K$ of characteristic zero whose elliptic fibration is non-isotrivial. Let $Z \subset X$ be a proper closed subset. Let $\Delta$ be an orbifold divisor on $X$ such that $(X,\Delta)$ is $1$-bounded modulo $Z$ and of general type. Then, for every smooth quasi-projective curve $C$ over $K$, the set of non-constant orbifold morphisms $f \colon C \to (X,\Delta)$ with $f(C) \not \subset Z$ is finite.  
\end{corollary}  
\begin{proof}  
It suffices to show that the scheme $\underline{\Hom}_K^{nc}(C,(X,\Delta)) \setminus \underline{\Hom}_K(C,Z)$ has only finitely many $K$-points. To do so, let $H$ be one of the finitely many irreducible components of $\underline{\Hom}_K^{nc}(C,(X,\Delta))\setminus \underline{\Hom}_K(C,Z)$. Then it suffices to show that $H(K)$ is finite. This is clear if $H$ is zero-dimensional. Thus, we may assume that $H$ is positive-dimensional in which case it is birational to a smooth projective curve of genus at least two (Theorem \ref{theorem:Kodaira_dimension_one}). It follows that $H(K)$ is finite by Faltings's theorem (\emph{formerly} Mordell's conjecture) \cite{FaltingsComplements}. This concludes the proof.
\end{proof}

\begin{remark}\label{remark:false}  
It follows from Remark \ref{remark:kloosterman} that Corollary \ref{cor:finiteness_for_orbifold_surface1} is false over algebraically closed fields of characteristic zero. More precisely, if $\overline{K}$ is an algebraic closure of $K$ (with notation as in Corollary \ref{cor:finiteness_for_orbifold_surface1}), then the set of non-constant orbifold maps $f \colon C_{\overline{K}}\to (X_{\overline{K}}, \Delta_{\overline{K}})$ with $f(C_{\overline{K}})\not\subset Z_{\overline{K}}$ is not necessarily finite.  
\end{remark}

\section{Bogomolov's theorem in the orbifold setting}\label{section:bogomolovs_thm}
Let $(X,\Delta)$ be a smooth projective orbifold pair over an algebraically closed field $k$ of characteristic zero. 
If $E \subset X$ is a closed subset, then we follow \cite{Dem97,  RTW, JR, Rou10, Rou12} and say that $(X,\Delta)$ is \emph{algebraically hyperbolic modulo $E$ over $k$} if, for every ample line bundle $L$ on $X$, there is a constant $\alpha_{X,\Delta,L,E}$ such that, for every smooth projective connected curve $\overline{C}$ over $k$, every dense open $C\subseteq \overline{C}$, and every non-constant orbifold morphism $f \colon C \to (X,\Delta)$ with $f(C)\not\subset \Delta\cup E$, the following inequality holds.
\[
\deg_{\overline{C}} \overline{f}^* L \leq \alpha_{X,\Delta,L,E} \cdot(2 \cdot g(\overline{C}) - 2 + \# (\overline{C}\setminus C)),
\] where $\overline{f}\colon \overline{C}\to X$ is the unique extension of $f$.
Obviously, if $(X,\Delta)$ is algebraically hyperbolic modulo $E$, then $(X,\Delta)$ is $1$-bounded modulo $E$.

We say that $(X,\Delta)$ is \emph{pseudo-algebraically hyperbolic over $k$} if there is a proper closed subset $E \subset X$ containing $\Delta$ such that $(X,\Delta)$ is algebraically hyperbolic modulo $E$.  

We start with a simple application of Riemann-Hurwitz to fibered surfaces.

\begin{lemma}\label{lemma:triviality}
Let $S$ be a smooth projective surface, let $D$ be an effective reduced divisor on $S$, let $B$ be a smooth projective curve, and let $p \colon S \to B$ be a flat proper (surjective) morphism. Then, there is a proper closed subset $Z\subsetneq S$ such that, for every smooth projective curve $C$ and non-constant morphism $f \colon C \to S$ contained in a fiber of $p$ satisfying $f(C)\not\subset Z$, the inequality
\[
\deg f^\ast (K_S+ D) \leq 2 g(C) - 2 + \#f^{-1}(D)
\]
holds.
\end{lemma}
\begin{proof} 
We may write $D = D^h + D^v$, where the components of $D^h$ are horizontal (i.e., surject onto $B$) and the components of $D^v$ are vertical (i.e., are contained in a fiber of $p$).
Let $B^{t}$ be the set of closed points $b$ in $B$ such that the intersection of $D^h$ and the fiber $p^{-1}(b) = S_b$ is transversal. Note that $B^t$ is a dense open of $B$.
Since $S$ is smooth, the fibration $p \colon S \to B$ has only finitely many singular fibers. 
We define $Z$ to be the (finite) union of $\supp D^v$, the singular fibers of $S \to B$, and the fibers $S_b$ where $b$ runs over all points $b \in B \setminus B^t$.

Let $C$ be a smooth projective curve and let $f \colon C \to S$ be a non-constant morphism with $f(C)\subset S$ contained in a fiber and $f(C) \not\subset Z$. By construction of $Z$, we have that $F:=f(C)$ is a smooth fiber of $p$, and that $F$ intersects $D$ transversally. Write $\iota \colon F \to S$ for the inclusion and $g \colon C \to F$ for the morphism induced by $f \colon C \to S$. Then, as $K_S|_F $ and $K_F$ are linearly equivalent as divisors on $F$, it follows that 
\[
\deg f^\ast (K_S + D) = \deg g^\ast K_F + \deg f^\ast D
\]
Since $F$ and $D$ intersect transversally, we have that $\iota^{-1}(D) = \iota^\ast D$. Therefore, 
\[
\deg g^\ast K_F + \deg f^\ast D = \deg g^\ast K_F + \deg g^\ast (\iota^{-1}(D))
\]
Note that the finite morphism $g \colon C\to F$ induces a morphism $(C,f^{-1}(D))\to (F,\iota^{-1}(D))$ of orbifolds (which, in this case, are log-pairs). In particular, there is an injective pullback morphism $g^\ast \omega_F(\iota^{-1}(D)) \to \omega_C(f^{-1}(D))$, so that 
\[ 
\deg g^\ast K_F + \deg g^\ast (\iota^{-1}(D)) = \deg g^\ast \omega_F(\iota^{-1}(D)) \leq \deg \omega_C(f^{-1}(D)) = \deg K_C + \#f^{-1}(D).
\]
This proves the lemma.
\end{proof}

Using Jouanolou's theorem \cite{Jou78} we get the following improvement of the previous lemma for foliations on surfaces.  

\begin{lemma}\label{lemma:log_case}
Let $S$ be a smooth projective surface, let $D$ be an effective reduced divisor on $S$, and let $\mathcal{F}$ be a foliation on $S$. Then, there exists a proper closed subset $Z \subsetneq S$ such that, for any smooth projective curve $C$ and any non-constant morphism $f \colon C \to S$ tangent to $\mathcal{F}$ with $f(C) \not \subset Z$, the following inequality holds.
\[
\deg f^*(K_S+D) \leq 2g(C) - 2 + \# f^{-1}(D)
\]
\end{lemma}
\begin{proof}
If $\mathcal{F}$ has only finitely many compact leaves, then we define $Z$ to be the union of these compact leaves, so that the statement is vacuously true. If $\mathcal{F}$ has infinitely many compact leaves, then Jouanolou's theorem \cite{Jou78} implies that $\mathcal{F}$ is a fibration in which case the statement follows from Lemma \ref{lemma:triviality}.
\end{proof}
 
Recall that, for $(X,\Delta)$ a smooth orbifold, Campana defined the sheaf of symmetric differentials $S^n \Omega_{(X,\Delta)} := S^n \Omega^1_{(X,\Delta)}$ \cite[Section 2.5]{Ca11}.  

The key insight of Bogomolov was that the existence of symmetric differentials on a smooth projective surface can be used to bound the degree of a curve in terms of its genus. As the following lemma shows, in the setting of orbifolds, one can instead use Campana's symmetric differentials to such an end.

\begin{lemma} \label{thm:bogomolov1} 
Let $(X,\Delta)$ be a smooth projective orbifold surface of general type over $k$ such that the following conditions hold.
\begin{enumerate}
\item For every ample line bundle $A$, there is an integer $n \geq 1$ such that $\mathrm{H}^0(X, S^n \Omega_{(X,\Delta)} \otimes A^{-1}) \neq 0$.
\item The multiplicities of $\Delta$ are all finite.
\end{enumerate} Then $(X,\Delta)$ is pseudo-algebraically hyperbolic over $k$.
\end{lemma}
\begin{proof} 
Since $\Delta$ has only finite multiplicities, to prove that $(X,\Delta)$ is pseudo-algebraically hyperbolic, it suffices to test on maps $C\to (X,\Delta)$ with $C$  smooth and  \emph{projective}. 

Let $D_0$ be an effective $\mathbb{Q}$-divisor such that $L_0 := K_X + \Delta - D_0$ is ample. Let $m$ be a positive integer such that $L:=mL_0$ is a divisor (with integer coefficients). By assumption, there exists an integer $n\geq 1$ and a non-zero orbifold symmetric differential $\omega$ in $\mathrm{H}^0(X, S^n \Omega_{(X,\Delta)} \otimes L^{-1})$.    

We consider the projectivized tangent bundle
\[
\pi \colon Y:=\mathbb{P}(T_X(-\log \lceil\Delta \rceil)) \to X.
\]
Note that $\omega$ corresponds to a global section of $\mathcal{O}_Y(n)\otimes \pi^*L^{-1}$. Let $S \subset Y$ be the zero divisor of $\omega$.  Let $D$ be the support of the restriction of  $\pi^\ast \lceil \Delta \rceil$ to $S$. 

Consider the tautological holomorphic foliation $\mathcal{F}$ of rank $1$ on $S$ induced by the subbundle $V \subset T_Y(-\log \pi^\ast \lceil\Delta \rceil)$ such that
\[
V_{x, [v]}:= \{\xi \in T_Y(-\log \pi^\ast \lceil\Delta \rceil) \ | \ d\pi(\xi) \in \mathbb{C}\cdot v\}.
\]
Let $\psi \colon \widetilde{S} \to S$ be a desingularization, let $\widetilde{D} $ be  the reduced divisor with support $\psi^{-1}(D)$, and let $\widetilde{\mathcal{F}}$ be the induced foliation on $\widetilde{S}$. Let $Z_1 \subset S$ be the exceptional locus of $\psi$.

If $C$ is a smooth projective curve and $f \colon C \to (X,\Delta)$ is a non-constant orbifold morphism such that $f^*\omega \neq 0$, then $f^*\omega$ is a non-zero global section of $S^n \Omega_C \otimes f^\ast L^{-1}$ \cite[Proposition 2.11]{Ca11}, so that 
\[
\deg f^\ast L \leq n \deg K_C = n(2g(C)-2).  
\]

Thus, it remains to treat the case where $f^*\omega = 0$. In this case, the morphism $f \colon C \to X$ factors over $\pi|_S \colon S \to X$ and is tangent to the foliation $\mathcal{F}$ on $S$ (i.e., is contained in a compact leaf of $\mathcal{F}$).

By Lemma \ref{lemma:log_case}, there exists a proper closed subset $\widetilde{Z} \subset \widetilde{S}$ such that, if $\widetilde{f} \colon C \to \widetilde{S}$ is a lift of $f \colon C \to X$ with $\widetilde{f}(C)\not\subset \widetilde{Z}$, then the following inequality holds.
\[
\deg \widetilde{f}^*(K_{\widetilde{S}}+\widetilde{D}) \leq 2g(C)-2+\#\widetilde{f}^{-1}(\widetilde{D})
\]
Define $Z$ to be the union of  $\pi(Z_1)$,  $\pi(\psi(\widetilde{Z}))$,  $ \supp D_0$, and the branch locus of $\widetilde{S}\to X$.  Note that $Z \subsetneq X$ is a proper closed subset of $X$. 

Now, for every morphism $f \colon C \to X$ with $f(C) \not\subset Z$, the inequality  
\[
\deg f^*(K_X+\lceil \Delta \rceil) \leq 2g(C) - 2 + \# f^{-1}( \lceil \Delta \rceil)
\]
holds. To prove this, note that such a morphism $f$ factors uniquely through a morphism $\widetilde{f} \colon C \to \widetilde{S}$ with $\widetilde{f}(C) \not \subset \widetilde{Z}$ and that, by Riemann--Hurwitz for surfaces, the difference \[K_{\widetilde{S}}+\widetilde{D} - (\pi\circ \psi)^\ast(K_X + \lceil \Delta\rceil)\] is effective.    In particular,   since  $Z$ contains the branch locus of  $\widetilde{S}\to X$, the pullback of this difference  along $\widetilde{f}$ has nonnegative degree on $C$.

Since $f \colon C \to (X,\Delta)$ is an orbifold morphism, we obtain 
\[
\deg f^*(K_X+\Delta) \leq 2g(C)-2.
\]
Finally, since $L_0 = K_X + \Delta - D_0$, we conclude that
\[
\deg f^* L_0 \leq \deg f^\ast (K_X+\Delta) \leq 2g(C) - 2.
\]
As $L=mL_0$, this implies 
\[
\deg f^* L \leq m(2g(C) - 2).
\]
This concludes the proof.
\end{proof} 

The following existence lemma for symmetric differentials is due to Bogomolov in the setting that the orbifold divisor is empty or logarithmic (i.e., all of its non-trivial multiplicities are infinite); see \cite[Corollary~10.11]{BogomolovHol}. If all multiplicities of $\Delta$ are finite, then the analogous existence result is proven in \cite[Corollary~5.4]{Rou12}. In general (when $\Delta$ has infinite and finite multiplicities), as we show now, the existence of symmetric differentials can be proven using a simple ``perturbation'' argument:

\begin{lemma} \label{thm:bogomolov2} 
If $(X,\Delta)$ is a smooth projective orbifold surface of general type satisfying $c_1(X,\Delta)^2 > c_2(X,\Delta)$ and $A$ is an ample line bundle on $X$, then there is an integer $n_0$ such that, for every $n \geq n_0$, we have that $\mathrm{H}^0(X, S^n \Omega_{(X,\Delta)} \otimes A^{-1}) \neq 0$.    
\end{lemma}
\begin{proof}
We may decompose $\Delta$ into a ``finite'' and ``infinite'' part. More precisely, if $\Delta = \sum_i (1-\frac{1}{m_i})\Delta_i$, we define $\Delta^{\log} = \sum_{i, m_i =\infty} \Delta_i$. Define $\Delta^{\mathrm{fin}} = \Delta - \Delta^{\log}$. (In other words, $\Delta^{\log} = \lfloor \Delta \rfloor$ and $\Delta^{\mathrm{fin}} = \Delta - \lfloor \Delta \rfloor$.)

Define $\Delta_m = \Delta^{\mathrm{fin}} + \left(1-\frac{1}{m}\right) \Delta^{\log}$. Since $K_X+\Delta$ is big, it follows that $K_X+\Delta_m$ is big for all sufficiently large $m$ (see \cite[Proposition~6.1]{BCJW}).  
Moreover, from the explicit formulas given in Definition \ref{definition:chernclasses}, it is clear that for all sufficiently large $m$, the inequality $c_1(X,\Delta_m)^2 > c_2(X,\Delta_m)$ continues to hold. Thus, for $m$ large enough, we have that $(X,\Delta_m)$ is of general type and satisfies $c_1(X,\Delta_m)^2 > c_2(X,\Delta_m)$. By applying \cite[Corollary~5.4]{Rou12} to $(X, \Delta_m)$, we have that $\mathrm{H}^0(X, S^n \Omega_{(X,\Delta_m)} \otimes A^{-1}) \neq 0$. Since $S^n \Omega_{(X,\Delta_m)} \subset  S^n \Omega_{(X,\Delta)}$, we obtain that $\mathrm{H}^0(X, S^n \Omega_{(X,\Delta)} \otimes A^{-1}) \neq 0$, as required.   
\end{proof}
 
We can now prove Bogomolov's theorem for orbifold surfaces stated below (Theorem \ref{thm:bogomolov} in the introduction).  As pointed out by McQuillan \cite[p.~122]{McQuillan}, this result is proven implicitly in Bogomolov's seminal paper \cite{Bogomolov, DeschampsBourbaki} when the orbifold divisor $\Delta$ is empty.

\begin{theorem}\label{thm:bogomolov_text}
If $(X,\Delta)$ is a smooth projective orbifold surface of general type  and $c_1(X,\Delta)^2>c_2(X,\Delta)$, then $(X, \Delta)$ is pseudo-algebraically hyperbolic over $k$.
\end{theorem} 
\begin{proof} 
If $\Delta$ has only finite multiplicities, combine Lemma \ref{thm:bogomolov1} and Lemma \ref{thm:bogomolov2}.  We now reduce to the case of finite multiplicities.  Indeed,    following the notation and arguments of  the proof of Lemma \ref{thm:bogomolov2},  we choose a large enough integer $m$ such that the ``perturbation''  $(X,\Delta_m)$ of $(X,\Delta)$  is    of general type and  satisfies $c_1(X,\Delta_m)^2 > c_2(X,\Delta_m)$. Since $\Delta_m$ has only finite multiplicities, it follows that $(X,\Delta_m)$, and thus $(X,\Delta)$,  is pseudo-algebraically hyperbolic.  
\end{proof}

\section{A cutting argument and the proof of Theorem \ref{thm:main2}}\label{section:proof_of_thm}

To conclude the proof of our    finiteness theorem for orbifold surfaces of general type with $c_1(X,\Delta)^2>c_2(X,\Delta)$, we take very general hyperplane sections to reduce to the case of curves:  

\begin{lemma}[Cutting argument]\label{lemma:cutting_argument}
Let $(X,\Delta)$ be a smooth projective orbifold over a finitely generated field $K$ of characteristic zero, and let $Z\subset X$ be a proper closed subset. Assume that, for every finitely generated field extension $L/K$ and every smooth quasi-projective curve $C$ over $L$, the set of non-constant orbifold maps $f \colon C\to (X_L,\Delta_L)$ with $f(C)\not\subset Z$ is finite. Then, for every finitely generated field extension $M/K$ and every smooth quasi-projective variety $V$ over $M$, the set of non-constant orbifold near-maps $f \colon V \to (X_M,\Delta_M)$ with $f(V)\not\subset Z_M$ is finite.
\end{lemma}
\begin{proof}
We combine the arguments in the proofs of \cite[Lemma~2.4]{JMRL} and \cite[Theorem~5.4]{BJ}. 

We argue by induction on $d:=\dim V$. If $d = 1$, then the required conclusion holds by assumption.  
Now, suppose that $d>1$. To prove the desired conclusion, let $M/K$ be a finitely generated field extension and $V$ a smooth quasi-projective variety over $M$ such that there is an infinite sequence of pairwise distinct orbifold maps $f_i \colon V\to (X_M,\Delta_M)$. Let $V\subset \mathbb{P}^n_M$ be an immersion. Let $M\subset \Omega$ be an uncountable algebraically closed field containing $M$,  and let $P\in V(\Omega)$ be such that $f_i(P)\neq f_j(P)$ for all $i\neq j$. Now, let $H\subset V_\Omega\subset \mathbb{P}^n_\Omega$ be a very general hyperplane section containing $P$. Since the restriction $f_i|_H \colon H \ratmap (X_\Omega, \Delta_\Omega)$ is defined at all points of codimension one of $H$ and does not factor over $\supp \Delta$, it follows from \cite[Lemma~2.5]{BJ} that $f_i|_H \colon H\ratmap (X,\Delta)$ is an orbifold near-map. Moreover, since $H\subset V_{\Omega}$ is very general, it follows that every near-map $f_i|_H$ is non-constant and that $f_i(H)\not\subset Z_{\Omega}$. We now descend $P$ and $H$ to a finitely generated extension of $M$.  Thus, let $M\subset M_2\subset \Omega$ be a finitely generated field extension of $M$ contained in $\Omega$ such that $P \in V(\Omega)$ lies in $V(M_2)$ and such that there is a hyperplane section $H_2\subset \mathbb{P}^n_{M_2}$ with $H_2\otimes_{M_2} \Omega = H$ (i.e., $H_2$ is a model for $H$ over $M_2$). Then, for every $i$, the morphism $f_i|_{H_2} \colon H_2 \ratmap (X_{M_2}, \Delta_{M_2})$ is a non-constant orbifold near-map with $f_i(H_2)\not\subset Z_{M_2}$. Also, the morphisms $f_i|_{H_2}$ are pairwise distinct (as they differ at $P$). As $\dim H_2 < d$, this contradicts the induction hypothesis and concludes the proof.
\end{proof}

\begin{corollary}\label{corollary:finiteness} 
Let $X$ be a Kodaira dimension one smooth projective surface over a finitely generated field $K$ of characteristic zero whose elliptic fibration is non-isotrivial. Let $Z\subset X_{\overline{K}}$ be a proper closed subset. Let $\Delta$ be an orbifold divisor on $X$ such that $(X,\Delta)$ is of general type and algebraically hyperbolic modulo $Z$. Then, for every finitely generated field extension $L/K$ and smooth quasi-projective variety $V$ over $L$, the set of orbifold near-maps $f \colon V \ratmap (X_L,\Delta_L)$ with $f(V_{\overline{L}}) \not \subset Z_{\overline{L}}$ is finite.  
\end{corollary}
\begin{proof}
Combine Corollary \ref{cor:finiteness_for_orbifold_surface1} and Lemma \ref{lemma:cutting_argument}.
\end{proof}

\begin{remark} 
The assumption that $K$ is finitely generated can not be dropped in Corollary \ref{corollary:finiteness}, as the desired conclusion fails over algebraically closed fields by Remark \ref{remark:false}.
\end{remark} 

\begin{theorem}\label{thm:Bdelta_general}
Let $(B,\Delta)$ be a smooth projective orbifold of general type over a finitely generated field $K$ of characteristic zero, where $B$ is a Kodaira dimension one surface with non-isotrivial elliptic fibration. If $c_1(B,\Delta)^2 > c_2(B,\Delta)$, then the following statements hold.
\begin{enumerate}
\item There is a proper closed subset $Z\subset B_{\overline{K}}$ such that the smooth projective orbifold $(B_{\overline{K}},\Delta_{\overline{K}})$ is algebraically hyperbolic modulo $Z$ over $\overline{K}$.
\item If $L/K$ is a finitely generated field extension and $V$ is smooth variety over $L$, then the set of orbifold near-maps $f \colon V \ratmap (B_L,\Delta_L)$ with $f(V_{\overline{L}}) \not \subset Z_{\overline{L}}$ is finite.  
\end{enumerate} 
\end{theorem}   
\begin{proof}
Combine Theorem \ref{thm:bogomolov} and Corollary \ref{corollary:finiteness}.
\end{proof}
Note that Theorem \ref{thm:main2} follows directly from (the slightly stronger) Theorem \ref{thm:Bdelta_general}.

\section{Bogomolov--Tschinkel's threefolds}\label{section:bt_threefolds}

The following lemma on BT-threefolds (Definition \ref{defn:bt}) was already proven in Bogomolov--Tschinkel \cite{BT}; we include a slightly different proof for the reader's convenience. Recall that, by definition, a BT-threefold fits into the Cartesian square
\begin{equation*} \begin{tikzcd}
X \ar[r] \ar[d] & S \ar[d, "\psi"] \\ B \ar[r, "\phi"] & \mathbb{P}^1
\end{tikzcd} \end{equation*}
where $S$ and $B$ are elliptic surfaces, $\psi$ is an elliptic fibration, and $\phi$ is a higher-genus fibration satisfying certain properties. Also, to state the lemma, recall that a variety $X$ over $k$ is \emph{algebraically simply-connected} if $\pi_1^{\et}(X_{\overline{k}})$ is trivial, where $\overline{k}$ is an algebraic closure of $k$. 

\begin{lemma}\label{lemma:bt}  
Let $X$ be a BT-threefold. Then $X$ is a smooth projective algebraically simply-connected threefold which is weakly-special but not special. Moreover, if $k=\mathbb{C}$, then $X^{\an}$ is simply-connected.
\end{lemma}
\begin{proof}
Over each point $t \in \mathbb{P}^1$ such that $\phi$ is smooth over $t$, the map $X \to S$ is smooth as well. Since $S$ is a smooth variety, it follows that each point of $X$ mapping to a point of $\mathbb{P}^1$ over which $\phi$ is smooth is itself a smooth point. The same argument holds for the map $\psi$. As the singular loci of $\phi$ and $\psi$ are disjoint, it follows that $X$ is smooth. The map $\phi$ has geometrically connected fibers, hence the map $X \to S$ does as well. As $S$ is connected, it follows that $X$ is connected. As $X$ is smooth, it follows that $X$ is integral, hence a variety. Since projective morphisms are stable under base change and composition, $X$ is projective. It is clear that $X$ is three-dimensional. 

Assume that $k=\mathbb{C}$. To see that $X$ is simply-connected, we may assume that $S\to\mathbb{P}^1$ is a relatively minimal elliptic fibration (as blow-ups of smooth projective varieties do not change the fundamental group). Now, we show that the open subset $U \subseteq X$ lying over $B \setminus D$ is simply-connected. There exists an elliptic surface $S' \to \mathbb{P}^1$ with $\chi(\mathcal{O}_S) = \chi(\mathcal{O}_{S'})$, which has a simply-connected fiber and a unique multiple fiber $F'$ of multiplicity $m$. By \cite[Theorem I.7.6]{FM94}, the surface $S$ is deformation equivalent to $S'$, and the deformation respects the elliptic fibrations of $S$ and $S'$. We assume that $F'$ lies over $0 \in \mathbb{P}^1$. The threefold $X$ is then deformation equivalent to $X' := B \times_{\mathbb{P}^1} S'$ and its open subset $U$ is deformation equivalent to $U' := (B \setminus D) \times_{\mathbb{P}^1\setminus\{0\}} (S' \setminus F')$.  Since $U$ and $U'$ are diffeomorphic, we have that $\pi_1(U) \cong \pi_1(U')$. Thus, it suffices to show that $U'$ is simply-connected. Now, the morphism $U' \to (B \setminus D)$ is an elliptic fibration with no multiple fibers and at least one simply-connected fiber. By \cite[Lemma 1.5.C]{Nori83}, the sequence
\[
\pi_1(F) \to \pi_1(U') \to \pi_1(B \setminus D) \to 1
\]
is exact, where $F \subseteq U'$ denotes any smooth fiber of $U' \to (B \setminus D)$. As $B \setminus D$ is simply-connected by definition, it hence suffices to show that the map $\pi_1(F) \to \pi_1(U')$ is the zero map. For this, let $F''$ be a simply-connected fiber of $U' \to (B \setminus D)$ and let $V \subseteq U'$ be an open neighborhood of $F''$ (for the Euclidean topology) which deformation retracts onto $F''$. In particular, $\pi_1(V)$ is trivial. Then $V$ contains a smooth fiber $F$ of $U' \to (B \setminus D)$. Thus, the map $\pi_1(F) \to \pi_1(U')$ factors over $\pi_1(V)$ and hence must be the zero map, as desired.

To show that $X$ is algebraically simply-connected, we may assume that $k=\mathbb{C}$, so that the result follows from the fact that $X^{\an}$ is simply-connected.

  To see that $X$ is weakly-special,    it suffices to show that $X$ does not admit a dominant rational map to any positive-dimensional variety of general type (as $X$ is simply-connected).   So assume that $X \ratmap Y$ is a dominant rational map to a positive-dimensional variety. In case $\dim Y = 3$, as $X$ is covered by a family of elliptic curves, the variety $Y$ will be as well.    In case $\dim Y = 2$, if the map $X \ratmap Y$ contracts the fibers of $X \to B$, it rationally factors over $X \to B$, so that $Y$ is rationally dominated by $B$. If $\dim Y = 2$ and $X \ratmap Y$ does not contract the fibers of $X \to B$, then $Y$ is again covered by a family of elliptic curves. Finally, if $\dim Y = 1$, then, as $X$ is simply-connected and hence has trivial Albanese variety, $Y$ must have trivial Albanese variety as well. Thus $Y \cong \mathbb{P}^1$. In any case, $Y$ is not of general type. 

To see that $X$ is not special, it suffices to observe that the orbifold base of the morphism $X \to B$ is $(B, (1-\frac{1}{m})D)$, which is of general type, as $m \geq 2$.   
\end{proof}
 
Since the fibers of $X \to B$ are elliptic curves, the core morphism of $X$ (as defined \cite[Definition~10.2]{Ca11}) must contract them, so that $(B, \Delta)$ is the core of $X$. We do not use this in what follows.

Despite the employed terminology, due to possible ``codimension two phenomena'', it is not at all clear that, given a fibration $f \colon X \to Y$ with orbifold base $\Delta_f$ the morphism $f \colon X \to Y$ induces an orbifold map $X \to (Y,\Delta_f)$; see \cite[Section~3.7]{JR} for a discussion. Fortunately, for a BT-threefold $X$, it is easy to see that $X\to (B,\Delta)$ is a flat orbifold morphism:
 
\begin{lemma}\label{lem:orbi_map}
If $X$ is a BT-threefold, then the morphism $X \to (B,\Delta)$ is flat and orbifold.
\end{lemma}
\begin{proof}
Since $X \to B$ is a base change of the flat morphism $S \to \mathbb{P}^1$, it follows that $X \to B$ is flat. In particular, the morphism $X \to B$ has no exceptional divisors. Thus, the morphism $X \to (B,\Delta)$ satisfies the orbifold condition by construction.
\end{proof}

The construction in \cite{BT} can be summarized by saying that BT-threefolds exist. We will need the following improvement of their result due to Campana--P\u{a}un \cite[Section 2]{CampanaPaun2007} (already alluded to in the introduction).

\begin{theorem}[Bogomolov--Tschinkel, Campana--P\u{a}un]\label{thm:btcp}
BTCP-threefolds exist. 
\end{theorem}
\begin{proof}
This result is proven in Sections 2.2 and 2.3 of \cite{CampanaPaun2007}. Indeed, in \emph{loc. cit.} the authors construct a smooth projective surface $S$ and a non-isotrivial elliptic fibration $\psi\colon S\to \mathbb{P}^1$ with precisely one multiple fiber $\psi^{-1}(0)$. Moreover, they construct a smooth projective surface $B$ of Kodaira dimension one and a morphism $\phi \colon B\to \mathbb{P}^1$ with the desired properties. The only condition that is not explicitly verified in \emph{loc. cit.} is that the elliptic fibration on $B$ is non-isotrivial. However, as their construction shows in Section 2.3 of \emph{loc. cit.}, the minimal model $B_0$ of $B$ can be taken to be a double cover of $\mathbb{P}^1\times \mathbb{P}^1$ ramified along any smooth divisor $R$ of type $(2(k+2), 4)$ with $k$ some sufficiently large integer. Choosing $R$ general enough, the resulting elliptic surfaces $B_0$ and $B$ are non-isotrivial.
\end{proof}

Our work culminates in the following results on Campana's conjecture for the general type orbifold surface $(B,\Delta)$.

\begin{corollary}\label{cor:Bdelta}
Let $K$ be a field of characteristic zero. If $X$ is a BTCP-threefold over $K$ with associated elliptic fibration $X \to (B,\Delta)$, then the following statements hold.
\begin{enumerate}
\item There is a proper closed subset $Z\subset B_{\overline{K}}$ such that the smooth projective orbifold $(B_{\overline{K}},\Delta_{\overline{K}})$ is algebraically hyperbolic modulo $Z$ and of general type over $\overline{K}$.
\item If $K$ is finitely generated over $\mathbb{Q}$, $L/K$ is a finitely generated field extension and $V$ is smooth variety over $L$, then the set of orbifold near-maps $f \colon V \ratmap (B_L,\Delta_L)$ with $f(V_{\overline{L}}) \not \subset Z_{\overline{L}}$ is finite.  
\end{enumerate}
\end{corollary}
\begin{proof}
Since the surface $B$ is of Kodaira dimension one with non-isotrivial elliptic fibration and the orbifold $(B,\Delta)$ is of general type, with $c_1(B,\Delta)^2 > c_2(B,\Delta)$ (as noted in Example \ref{example:btcp}), the corollary follows from Theorem \ref{thm:Bdelta_general}.
\end{proof}

\begin{remark}  
  Corollary \ref{cor:Bdelta}.(2) is false over algebraically closed fields by Remark \ref{remark:false}.  
\end{remark}

\begin{lemma}\label{lemma:faltings2}
Let $V$ be a variety over a finitely generated field $K$ of characteristic zero. 
Then the set of $K$-isomorphism classes of abelian varieties  dominated by $V$ is finite.
\end{lemma}
\begin{proof}  
Let $\overline{V}$ be a smooth projective variety birational to $V$. If $V$ dominates an abelian variety $B$, then there is a surjective morphism $\overline{V} \to B$ (as $B$ has no rational curves) and thus, by the universal property of Albanese varieties,  the Albanese variety $A:=\mathrm{Alb}(\overline{V})$ of $\overline{V}$ dominates $B$.  
Up to isogeny, $A$ is a product of simple abelian varieties $A_1, \ldots, A_n$. If $A$ dominates an abelian variety $B$, then $B$ is isogenous to $\prod_{i \in I} A_i$ with $I \subseteq \{1,...,n\}$, so that there are only finitely many isogeny classes that $B$ could be in. However, by Faltings's Isogeny Theorem \cite{FaltingsComplements}, the set of $K$-isomorphism classes of abelian varieties $B$ over $K$ which are isogenous to a fixed abelian variety (over $K$) is finite.
\end{proof}

\begin{lemma} \label{lemma:isotriviality} 
Let $B$ be a variety over a finitely generated field $K$ of characteristic zero and let $X\to B$ be an elliptic fibration. Let $V$ be a variety over $K$. If the set of $b$ in $B(K)$ such that $V$ dominates $X_b$ is dense in $B$, then $X \to B$ is isotrivial.  
\end{lemma}
\begin{proof}
By Lemma \ref{lemma:faltings2}, $V$ dominates only finitely many   elliptic curves over $K$ (up to isomorphism). Let $B^{\circ} \subseteq B$ be the smooth locus of $X \to B$. Let $j \colon B^\circ \to \mathbb{A}^1$ be the moduli map (or $j$-invariant) of the Jacobian of $X|_{B^\circ} \to B^\circ$. Since there is a dense subset of $B$ over which all fibers are pairwise isomorphic, the morphism $j$ is constant. 
\end{proof}

We are now ready to prove our main   non-density result for BTCP-threefolds (Theorem \ref{thm:main_final}) which we now restate for the reader's convenience.

\begin{theorem} \label{thm:main_final_text}
Let $X$ be a BTCP-threefold over a finitely generated field $K$ of characteristic zero and let $V$ be a variety over $K$. Then the set of non-constant rational maps $V \ratmap X$ is not dense in $X_{K(V)}$.   
\end{theorem} 
 
\begin{proof}  
Let $V$ be a variety over $K$. Assume that we have a sequence $f_i \colon V \ratmap X$ of orbifold near-maps which are dense in $X_{K(V)}$.
Let $\pi\colon X \to (B,\Delta)$ be the associated elliptic fibration. It follows from Corollary \ref{cor:Bdelta}.(2) that the set of non-constant orbifold near-maps $V \ratmap (B,\Delta)$ is not dense in $B_{K(V)}$. Since $X \to (B,\Delta)$ is a flat orbifold map (Lemma \ref{lem:orbi_map}), we obtain that the set of non-constant near-maps $V \ratmap X$ for which $V \ratmap X \to B$ is non-constant is not dense in $X_{K(V)}$. Therefore, replacing $f_i$ by a suitable subsequence, we may assume that each composition $\pi \circ f_i \colon V \ratmap B$ is constant. Let $b_i$ be the image of $\pi \circ f_i$. Since $V$ is geometrically connected over $K$, it follows that $b_i$ is a $K$-point of $B$. Moreover, since the set $\{b_i \ | \ i\in\mathbb{Z}_{\geq 1}\}$ is dense in $B$ and $V$ dominates $X_{b_i}$ for every $i$, we obtain that $X \to B$ is isotrivial (Lemma \ref{lemma:isotriviality}) which contradicts    the fact that $X \to B$ is non-isotrivial (by definition).
\end{proof}

\begin{remark}
Let $X$ be a BTCP-threefold over a finitely generated field $K$ of characteristic zero and let $\pi \colon X \to (B,\Delta)$ be the associated elliptic fibration. Let $Z \subset B_{\overline{K}}$ be as in Corollary \ref{cor:Bdelta}. Then, the proof of Theorem \ref{thm:main_final} shows that, for $V$ a variety over $K$, for all but finitely many non-constant rational maps $f \colon V \ratmap X$, the image of $f$ is contained in a fibre of $X \to B$ over a $K$-point of $B$ or in $\pi^{-1}Z$.   
\end{remark}

\begin{remark}  
Theorem \ref{thm:main_final} is false over algebraically closed fields. Indeed, for \emph{every} BTCP-threefold $X$ over an algebraically closed field $k$ of characteristic zero and every elliptic curve $E$ over $k$ with function field $K$, one can show that the set of morphisms $E \to X$ is dense in $X_K$. We conclude that the conclusion of Theorem \ref{thm:btcp} is optimal.
\end{remark}

\subsection{Geometrically special varieties}\label{section82}
In this section, we assume that $k$ is algebraically closed of characteristic zero.
	
Recall that $(X,\Delta)$ is pseudo-$1$-bounded (over $k$) if there is a proper closed subset $E\subsetneq X$ containing $\Delta$ such that $(X,\Delta)$ is $1$-bounded modulo $E$. If $\Delta=\emptyset$, then it is not hard to show that pseudo-$1$-boundedness implies finiteness of pointed maps using bend-and-break. However, in the more general orbifold setting (with $\Delta$ not necessarily empty), we can only show the non-density of pointed maps, assuming in addition two-dimensionality of $X$, non-uniruledness of $X$, and bigness of $K_X+\Delta$. Let us be more precise.
 
Following \cite[Definition~3.12]{JR}, an orbifold $(X,\Delta)$ over $k$ is \emph{geometrically-special over $k$} if, for every dense open subset $U \subseteq X$, there exists a smooth projective connected curve $C$ over $k$, a point $c$ in $C(k)$, a point $u$ in $U(k)\setminus \Delta$, and a sequence of pairwise distinct \textbf{orbifold} morphisms $f_i \colon C\to (X,\Delta)$ with $f_i(c) = u$ for $i=1,2,\ldots$ such that $\cup_i \Gamma_{f_i}$ is dense in $C \times X$.

\begin{corollary} \label{cor:surface_not_gs}
Let $(B,\Delta)$ be a smooth projective orbifold surface with $B$ non-uniruled. If $(B,\Delta)$ is of general type and pseudo-$1$-bounded, then $(B,\Delta)$ is not geometrically-special.
\end{corollary} 
\begin{proof}  
Let $Z\subsetneq B$ be a proper closed subset such that $(B,\Delta)$ is $1$-bounded modulo $Z$.  We argue by contradiction. Assume $(B,\Delta)$ is geometrically-special. Then, there is a point $x$ in $B\setminus Z$, a smooth proper curve $C$, and a point $c$ in $C(k)$ such that the universal evaluation map 
\[
C \times \underline{\Hom}^{nc}((C,c), ((B,\Delta),x)) )\to C\times B
\]
is dominant. Note that $H := \underline{\Hom}^{nc}((C,c), ((B,\Delta),x))$ is a closed subscheme of \[\underline{\Hom}^{nc}(C,(B,\Delta)) \setminus \underline{\Hom}(C,Z),\]  as it is given by the scheme-theoretic fibre of the morphism $\mathrm{ev}_c\colon \underline{\Hom}(C,(B,\Delta))\to B$ over $x$. Since  $$\underline{\Hom}^{nc}(C,(B,\Delta)) \setminus \underline{\Hom}(C,Z)$$ is of finite type over $k$, it follows that $H$ is of finite type. In particular, since $C \times H \to C \times B$ is a dominant morphism of finite type schemes over $k$, we have that $\dim H \geq \dim B = 2$. However, by our assumptions and Corollary \ref{corollary:boundedness},  the scheme $H$ is at most one-dimensional. This contradiction completes the proof.
\end{proof}

\begin{theorem}\label{thm:bt}
If $X$ is a BTCP-threefold, then $X$ is not geometrically-special.  
\end{theorem}
\begin{proof}  
Suppose that $X$ were geometrically-special. Consider the associated elliptic fibration $X \to (B,\Delta)$. Since $X \to (B,\Delta)$ is a surjective orbifold morphism (Lemma \ref{lem:orbi_map}), the orbifold pair $(B,\Delta)$ is geometrically-special \cite[Lemma~3.14]{JR}. However, since $B$ is non-uniruled and $(B, \Delta)$ is a pseudo-algebraically hyperbolic orbifold surface of general type (Corollary \ref{cor:Bdelta}), the orbifold  $(B,\Delta)$ is not geometrically-special (Corollary \ref{cor:surface_not_gs}). This contradiction concludes the proof.  
\end{proof}
%
%

\bibliographystyle{alpha}
\bibliography{orbi}{}
\vfill
\end{document}